\documentclass[10pt]{amsart}

\usepackage[latin2]{inputenc}
\usepackage{amsmath}
\usepackage{graphicx}
\usepackage{amssymb}
\usepackage{esint}
\usepackage{color}
\usepackage{amsthm}
\usepackage{epsfig}
\usepackage{enumitem}
\usepackage{mathtools}
\usepackage{esvect}
\usepackage{tikz}
\usepackage{mathrsfs}
\usetikzlibrary{calc,intersections,through,backgrounds,patterns,arrows.meta, math,through}
\usepackage[english]{babel}
\usepackage[linktocpage=true,colorlinks=true,linkcolor=Blue,citecolor=Green]{hyperref}

\newtheorem{theorem}{Theorem}
\newtheorem{proposition}[theorem]{Proposition}
\newtheorem{lemma}[theorem]{Lemma}

\newtheorem*{theorem*}{Theorem}

\theoremstyle{definition}
\newtheorem{definition}{Definition}
\newtheorem{remark}[definition]{Remark}

\def\XXint#1#2#3{{\setbox0=\hbox{$#1{#2#3}{\int}$ }
\vcenter{\hbox{$#2#3$ }}\kern-.6\wd0}}

\definecolor{Yellow}{rgb}{0.95,0.9,0.0} 
\definecolor{Red}{rgb}{0.8,0.1,0.1}
\definecolor{Green}{rgb}{0.1,0.65,0.2}
\definecolor{Blue}{rgb}{0.1,0.1,0.8}
\definecolor{Purple}{rgb}{0.7,0.1,0.7}
\definecolor{Grey}{rgb}{0.6,0.6,0.6}

\definecolor{YELLOW}{rgb}{0.95,0.9,0.0} 
\definecolor{RED}{rgb}{0.8,0.1,0.1}
\definecolor{GREEN}{rgb}{0.25,0.65,0.1}
\definecolor{BLUE}{rgb}{0.1,0.1,0.8}
\definecolor{PURPLE}{rgb}{0.7,0.1,0.7}

\newcommand{\supp}{\operatorname{supp}} 
\renewcommand{\div}{\operatorname{div}}

\DeclareMathOperator*{\esssup}{ess\,sup}

\renewcommand{\epsilon}{\varepsilon}

\allowdisplaybreaks[4]

\newcommand{\dx}{\mathrm{d}x}
\newcommand{\dt}{\mathrm{d}t}

\newcommand{\eps}{\varepsilon}

\renewcommand{\vec}[1]{{\operatorname{#1}}}
\newcommand{\n}{\vec{n}}

\begin{document}

\title[A phase-field model for the Verigin problem] 
{A phase-field model for the Verigin problem: weak solutions and sharp-interface limit}

\author{Anna Kubin}
\address{Institut f\"ur Analysis und Scientific Computing, Technische Universit\"at Wien, Wiedner Hauptstrasse 8-10, 1040 Vienna, Austria}
\email{anna.kubin@tuwien.ac.at}

\author{Alice Marveggio}
\address{Institute for Applied Mathematics, Universit{\"a}t Bonn, Endenicher Allee 62, 53115 Bonn, Germany}
\email{alice.marveggio@hcm.uni-bonn.de}

\begin{abstract} 
We introduce a novel phase-field model approximating the Verigin problem with phase
transition. 
For fixed interface thickness $\eps>0$, we construct weak solutions through
a time-discrete variational scheme combining a Wasserstein minimizing-movement
for the density with Sch\"atzle's selection procedure for the phase field,
which yields strong compactness and an almost-minimizing property.
We then investigate the sharp-interface limit $\eps \downarrow0$. 
For well-prepared initial data, we establish strong compactness of the densities and phase fields, and we prove convergence of the time-integrated diffuse interfacial energy to the perimeter. The limiting evolution is identified as a weak solution of the
Verigin problem with phase transition, satisfying an optimal energy dissipation inequality.

\vspace{5pt}
    
    \vspace{5pt}
    \noindent
    \textbf{Mathematical Subject Classification (MSC) 2020:} 
    35R35, 
    35B25, 
    49Q22, 
    76S05, 
    35D30. 
\end{abstract} 


\maketitle

\section{Introduction}
The Verigin problem describes the evolution of a compressible two-phase flow in a porous medium. A typical situation arises when a porous medium containing a compressible fluid is flooded with a second compressible fluid, which displaces the former under isothermal conditions. If phase transitions between the two phases are also taken into account, the resulting system is known as the Verigin problem with phase transition, a free-boundary problem that can be regarded as the compressible counterpart of the Muskat problem.

To the best of our knowledge, the mathematical literature on the Verigin problem with phase transition has until recently been restricted to the theory of classical smooth solutions (see~\cite{PrussSim,PrussSimWielke} and the references therein). In particular, short-time existence, stability, and convergence to equilibrium for solutions that do not develop singularities are studied in~\cite{PrussSim}. More recently, in a joint work with Tim Laux~\cite{KLM}, we established the first global-in-time existence result for weak solutions to the 
Verigin problem with phase transition.
In particular, we constructed distributional solutions through an implicit time-discretization scheme exploiting the gradient-flow structure of the system.

In this manuscript, we introduce a novel phase-field model approximating the Verigin problem with phase transition and study both the existence of weak solutions and its sharp-interface limit. The starting point is the underlying energy of the system, which is the phase-field analogue of the energy associated with the sharp-interface Verigin problem.
For $\Omega \subset \mathbb{R}^d$ open set with smooth boundary and $T>0$ finite time horizon, denoting by $\rho: \Omega \times [0,T) \rightarrow [0,\infty)$ the total density of the two-phase fluid and by $\varphi: \Omega \times [0,T) \rightarrow \mathbb{R}$ the phase-field variable, at each time $t \in (0,T)$ the total energy is given by
\begin{equation} \label{eq:energyintro}
    \mathcal{E}_\eps(\rho,\varphi):=  \int_{\Omega} \frac\varepsilon2 |\nabla \varphi|^2 + \frac1\varepsilon F(\varphi) \, \mathrm{d}x   +   \int_{\Omega}   \Big [ \tfrac{1}{m-1}	 \rho^m +  \varphi ( \lambda   - \rho) \Big] \, \mathrm{d}x, \quad m > 1,
\end{equation}
where $\eps>0$ represents the interface thickness and $F$ is a standard double-well potential, i.e., $F(s)=\frac{1}{4}s^2(s-1)^2$ for $s \in \mathbb{R}$. The first term in~\eqref{eq:energyintro} is the Cahn--Hilliard energy~\cite{CahnHilliard} associated with the diffuse interface, which is well known to $\Gamma$-converge to the perimeter as $\eps \downarrow 0$~\cite{ModicaMortola}.
The second term in~\eqref{eq:energyintro} is the bulk free energy of the density, while the last term couples the phase and density variables accounting for the phase transition. 

Phase-field models replace the evolving sharp interface by a thin diffuse transition layer of width $\eps >0$, thereby avoiding the explicit tracking of the sharp interface (or free boundary) and providing a formulation on the whole spatial domain. 
Besides providing a convenient analytical formulation, this approach is particularly advantageous from a numerical point of view, since diffuse-interface models can be approximated by standard discretization schemes for PDEs and naturally accommodate changes in the topology of the evolving interface.

Our first main result (cf.~Theorem~\ref{theo:existence}) concerns the existence of weak solutions to the diffuse-interface model for $m \geq 2$ and $d \in \{2,3\}$ (cf.~Remark~\ref{rk:existence} and Remark~\ref{remarkd}).
We construct such weak solutions through an implicit time discretization scheme exploiting the variational structure of the system. 
At each time step, we first select a phase-field equilibrium associated with the previous density, 
and then perform a Wasserstein minimizing movement for the density with the selected phase field kept fixed. The selection mechanism ensures the decay of the phase-field energy, 
strong compactness of the phase field, and an almost-minimizing property for the phase-field energy in the limit. 
Passing to the continuous-time limit, we obtain weak solutions to a phase-field system, whose strong formulation, for $\eps >0$, reads 
\begin{align} \label{eq:phasefield}
\begin{cases}
    \partial_t \rho_\eps + \div (\rho_\eps \mathbf{u}_\eps) = 0 \\
   \rho_\eps \mathbf{u}_\eps  = \div \mathbb{T}_\eps - \nabla (\rho_\eps^m - \lambda \varphi_\eps) \\
   \mathbb{T}_\eps = \Big(\frac{\eps}{2} |\nabla \varphi_\eps|^2 + \frac{1}{\eps} F(\varphi_\eps)  \Big) \operatorname{Id}  - \eps \nabla \varphi_\eps\otimes  \nabla \varphi_\eps \\
 - \eps \Delta \varphi_{\epsilon} + \frac1\eps F^\prime(\varphi_{\epsilon}) + \lambda-\rho_{\epsilon}= {0}
   \end{cases}
\end{align}
on $\Omega \times (0,T)$. This system is endowed with Neumann boundary conditions for both $\rho_\eps$ and $\varphi_\eps$ and with initial conditions $\rho_\eps (\cdot, 0)= \rho_0$ and $\varphi_\eps(\cdot, 0) = \varphi_0 $.
Here, $ \mathbf{u}_\eps $ is the velocity field and $ \mathbf{j}_\eps := \rho_\eps \mathbf{u}_\eps$ is the mass flux degenerating in $\{ \rho_\eps =0 \}$.

Our construction is closely related to variational approaches to dissipative
evolution equations. The density evolution is constructed by a
Wasserstein minimizing-movement scheme in the spirit of~\cite{JKO,Otto1,Chambolle-Laux,JKM_Muskat}, whereas the phase-field variable is required to remain at equilibrium at each time step, meaning that $\varphi_\eps$ satisfies a stationary Allen-Cahn equation with $\rho_\eps$ as a forcing term. The latter equilibrium problem is
nonconvex and may admit several solutions, so that an arbitrary choice of
equilibria would not provide the strong compactness needed in the limit.
To overcome this issue, we adapt a selection argument adopted by Sch\"atzle~\cite{Schaetzle} in his analysis of a quasistationary phase-field model of Caginalp type. This selection procedure allows us to follow suitable branches of equilibria
while preserving the decay of the phase-field energy. 
For the density variable, strong compactness follows by combining the   control in time  provided by the Wasserstein distance with the compactness in measure criterion  of Rossi and Savar\'e~\cite{Rossi-Savare}.  More generally, this approach fits into the variational framework for the stability and convergence of gradient-flow evolutions (see, e.g.,~\cite{RossiSavare2006,MielkeRossiSavare}).
The Wasserstein variational approach also plays a central role in the work of Kroemer and Laux~\cite{KroemerLaux} (see also~\cite{ElbarPoiatti}) for the Cahn--Hilliard equation with disparate mobilities,
where the underlying gradient-flow structure is exploited to construct weak
solutions.

Our second main result  (cf.~Theorem~\ref{theo:sharplimit}) concerns the study of the sharp-interface limit $\eps\downarrow0$. 
We prove strong convergence of $\rho_\eps$ to a density $\rho$ and of the phase field $\varphi_\eps$ to a characteristic function $\chi$, as the parameter $\eps$ vanishes. Moreover, we identify a limiting velocity field $\mathbf{u}:\Omega \times [0,T) \rightarrow \mathbb{R}^d$ such that $(\rho,\chi,\mathbf{u})$ is a distributional solution of the Verigin problem with phase transition: 
\begin{align}
\begin{cases}
	\partial_t \rho + \div(\rho \mathbf{u})=  0  & \; \text{ in } \Omega , \\
    [\![\rho]\!]V_\Gamma =	[\![ \rho  \mathbf{u} \cdot \n_\Gamma ]\!] & \;  \text{ along } \Gamma , 
\end{cases}
\quad \quad 
\begin{cases}  
    \rho \mathbf{u}= - \nabla \pi    & \; \text{ in } \Omega ,  \\
	[\![\pi]\!]= \sigma H  &\;  \text{ along } \Gamma,
\end{cases}
\end{align}    
with boundary conditions
\begin{align} \notag 
       \nabla\rho \cdot \n_{\partial \Omega} &= 0  \; \; \;   \text{ along } \partial \Omega, \\ \notag 
        \n_\Gamma \cdot \n_{\partial \Omega} &= 0 \; \; \; \text{ along } \partial\Omega \cap  \Gamma,
\end{align}
and with initial data $\chi(\cdot , 0)= \chi_0 $ and $\rho(\cdot,0)=\rho_0$.
Here, $\pi := \rho^m - \lambda \chi$ denotes the pressure, $\sigma>0$ is the surface tension, 
 $	[\![\cdot ]\!] $ denotes the jump across the interface $\Gamma := \partial \{ \chi = 1\}$, $V_\Gamma$ is the velocity of $\Gamma$ in the direction of the normal $\nu_\Gamma$, and $H_\Gamma$ is the (scalar) mean curvature of $\Gamma$.

Starting from the seminal
$\Gamma$-convergence result of Modica and Mortola~\cite{ModicaMortola}, substantial work has been devoted to understanding whether the convergence of energies can be extended from the static setting to the associated gradient-flow evolution~\cite{SandierSerfaty}. 
For convergence towards smooth interface evolutions, we refer to the classical works~\cite{Chen,AlikakosBatesChen,DeMottoniSchatz,CaginalpChen}, as well as to subsequent extensions covering anisotropic~\cite{AlfaroGarckeHilhorstMatanoSchaetzle} and boundary effects~\cite{AbelsMoser}, and to more recent quantitative results~\cite{FischerLauxSimon,FischerMarveggio}.
For the (evolutionary) Allen--Cahn equation,
the sharp-interface limit leads to mean curvature flow, and convergence has been
established for several notions of weak geometric evolution. 
These include viscosity solutions~\cite{EvansSonerSouganidis}, 
Brakke's motions~\cite{Ilmanen,MizunoTonegawa}, as well as De Giorgi-type notions~\cite{HenselLauxvarifold,Steinke} and $BV$ formulations~\cite{LauxSimon,HenselLauxcontact}. Extensions to weighted and anisotropic mean curvature flows can be found in~\cite{GanediMarveggioStinson,LauxUllrichStinson}.
One of the main advantages of weak solution concepts is that they allow for a description of the evolution beyond the formation of singularities.
Among early sharp-interface convergence results for phase-field models,
Sch\"atzle~\cite{Schaetzle} studied a quasistationary phase-field
model of Caginalp type and proved its convergence
to weak solutions of the Stefan problem with Gibbs--Thomson law in the sense of~\cite{Luckhaus1,Luckhaus0}. 
Related results on the rigorous derivation of the Gibbs--Thomson law from the first variation of the diffuse-interface energy can be found in~\cite{LuckhausModica,RogerTonegawa}.
For conserved phase-field
dynamics, we recall Chen's global-in-time convergence result for the Cahn--Hilliard equation towards a weak varifold formulation of the Mullins--Sekerka flow~\cite{Chen96}. More recently, Kroemer and Laux~\cite{KroemerLaux} also proved the convergence of the
Cahn--Hilliard equation with disparate mobilities to the (weak) Hele--Shaw flow,
under a suitable convergence assumption for the diffuse interfacial energies. The anisotropic and inhomogeneous setting was more recently addressed by Elbar and Poiatti~\cite{ElbarPoiatti}.

A key point in the sharp-interface analysis is the convergence of the interfacial energies. 
The almost-minimizing property obtained for the weak solutions of~\eqref{eq:phasefield} allows us to prove the convergence of the time-integrated 
 diffuse interfacial energies to the perimeter as the interface thickness $\eps>0$ vanishes. 
This reflects the quasi-static structure underlying the Verigin problem (cf.~\cite{KLM}), which is also characteristic of the Stefan problem with Gibbs--Thomson law (cf.~\cite{Luckhaus1}), and, consequently, inherited by the corresponding phase-field approximation.
As a consequence, we obtain distributional, or equivalently $BV$, solutions of the Verigin problem with phase transition without imposing an \textit{a priori} convergence assumption on the interfacial energies, in the same spirit as in~\cite{KLM} (and in~\cite{Luckhaus0}). Such an assumption was introduced in the seminal work~\cite{LuckStur} and frequently appears in the literature under the name of ``energy convergence hypothesis'', 
both for Allen--Cahn-type approximations~\cite{LauxSimon,HenselLauxcontact,LauxUllrichStinson,GanediMarveggioStinson,Steinke} and for conserved phase-field models~\cite{KroemerLaux,LauxStinson,ElbarPoiatti}.
In contrast, this hypothesis is not needed in weaker varifold formulations (see, e.g.,~\cite{Ilmanen,HenselLauxvarifold,LauxTakasao}), where the limiting interface may carry nontrivial multiplicity.

The rest of the paper is organized as follows. 
In Section~\ref{sec:existence}, we introduce the phase-field approximation of the Verigin problem with phase transition and we establish the existence of weak solutions. More precisely, we first introduce the time-discrete variational scheme (cf.~Section~\ref{sec:minprob}), then derive the compactness properties of the discrete solutions (cf.~Section~\ref{sec:compactness}), including the strong compactness and almost-minimizing property of the phase field, and finally pass to the continuous-time limit (cf.~Section~\ref{sec:continuouslimit}). 
Section~\ref{sec:sharplimit} is devoted to the analysis of the sharp-interface limit $\eps\downarrow0$. We first prove strong compactness of the density and phase field, identify the limiting velocity (cf.~Section~\ref{sec:compsharp}), then establish the convergence of the time-integrated diffuse interfacial energy to the perimeter energy (cf.~Section~\ref{sec:energyconvsharp}), and finally show that the limiting triple is a distributional solution to the Verigin problem with phase transition (cf.~Section~\ref{sec:limitsharp}). 
Finally, Appendix~\ref{appendixA} recalls a compactness  in measure result from~\cite{Rossi-Savare}, while Appendix~\ref{appendixB} collects some auxiliary technical results from~\cite{Schaetzle}.

\medskip

\paragraph{\bfseries Notation}
Let $d$ be an integer number larger than or equal to $2$, and let $\Omega \subset \mathbb{R}^d$ be a smooth open set with boundary $\partial \Omega$ and outer normal vector field $\nu_\Omega$. 
We denote by $\mathcal{M}_b$ the set of bounded Radon measures, and $\mathcal{L}^d$ denotes the $d$-dimensional Lebesgue measure. 
For  $\mu \in \mathcal{M}_b(U)$ and $A \subset U$, $\mu\llcorner A$ is the restriction of $\mu$ to $A$. 
With a slight abuse of notation, for probability densities $\rho$, $\tilde \rho$ on $\Omega$, we denote by $W_2(\rho, \tilde \rho)$  the $2$-Wasserstein distance between the measures $\rho \mathcal{L}^d\llcorner\Omega$ and $\tilde \rho \mathcal{L}^d\llcorner\Omega$, and, given $F:\Omega \to \Omega$, we denote by $F_\# \rho$ the pushforward of $\rho \mathcal{L}^d\llcorner\Omega$ by $F$.
For any measurable set $E \subset \Omega$ we denote by $\chi_E$ its characteristic function.
We say that $E$ is a set of finite perimeter in $\Omega$ if $\chi_E\in BV(\Omega;\{0,1\}),$ and we denote its distributional derivative in $\Omega$ by $\nabla \chi_E$ and its total variation measure by $|\nabla \chi_E|$. 
Denoting by $\partial^\ast E$  the reduced boundary of the set $E$, we have $\int_\Omega \cdot \, \mathrm{d} |\nabla \chi_E| =  \int_{\partial^\ast E \cap \Omega} \cdot \, \mathrm{d} \mathcal{H}^{d-1}$, where $\mathcal{H}^{d-1}$ is the $(d-1)$-Hausdorff measure.
We denote by $L^1_0(\Omega)$ the quotient of $L^1(\Omega)$ by constant functions, namely
$L^1_0(\Omega):=L^1(\Omega)/\mathbb{R}$, endowed with the norm $\|\cdot\|_{L^1_0(\Omega)}:= \inf_{c \in \mathbb{R}} \|\cdot-c\|_{L^1(\Omega)}.$
Given a function space $X$ and functional $\mathcal{F}:X \to \mathbb{R}$, we denote by $\delta \mathcal{F}(\varphi) \in X'$ the first variation of $\mathcal{F}$ at $\varphi \in X$.
We denote by $\operatorname{Id}$ the identity matrix in $\mathbb{R}^{d \times d},$ for $a, b \in \mathbb{R}^d$ we set $a \otimes b:=(a_i b_j)_{ij} \in \mathbb{R}^{d \times d}$, and for $A, B \in \mathbb{R}^{d \times d}$ we set $A :B = A_{ij}B_{ij} \in \mathbb{R}.$

\section{Existence of weak solutions to the phase-field model}
\label{sec:existence}
Let $\Omega \subset \mathbb{R}^d$ be a 
bounded connected open set with smooth boundary, $\lambda \in (0,\infty)$ and $m \in [2,\infty)$.
For any measurable pair of functions $(\rho,\varphi): \Omega \to [0,\infty) \times \mathbb{R}$, we consider the total energy 
\begin{equation}\label{phasefieldenergy}
    \mathcal{E}_\eps(\rho,\varphi):=  \mathcal{E}_{\textnormal{int}, \eps}(\varphi)  +   \int_{\Omega}   \Big [ \tfrac{1}{m-1}	 \rho^m +  \varphi ( \lambda   - \rho) \Big] \, \mathrm{d}x, \quad \eps>0,
\end{equation}
where the diffuse interfacial energy $\mathcal{E}_{\textnormal{int}, \eps}(\varphi) $ is given by 
\begin{equation*}
\mathcal{E}_{\textnormal{int}, \eps}(\varphi) :=   \int_{\Omega} \frac\varepsilon2 |\nabla \varphi|^2 + \frac1\varepsilon F(\varphi) \, \mathrm{d}x \end{equation*}
and $F$ is a standard double-well potential, i.e., $F(s)=\frac{1}{4}s^2(s-1)^2$ for $s \in \mathbb{R}$.

Before stating our main existence result, we provide the notion of weak solution for~\eqref{eq:phasefield} as follows.
\begin{definition}[Weak solution to the phase-field model] 
\label{def:weaksol}
Let $\varepsilon >0$. Given initial data $\varphi_0 \in H^1(\Omega)$ and $\rho_0 \in L^m(\Omega)$ such that $\int_\Omega \rho_0 \, \dx =1 $, 
we say that $(\rho_\eps, \varphi_\eps, \mathbf{u}_\eps)$ is a weak solution to the phase-field model~\eqref{eq:phasefield} if the following properties hold.
\begin{itemize}
\item[i)] 
$\rho_\eps \in L^m(\Omega \times (0,T))$, $\varphi_\eps \in L^4(\Omega \times (0,T)) \cap L^2(0,T;H^1(\Omega))$ and $\mathbf{u}_\eps \in (L^2( \Omega \times (0,T);\rho_\eps \,\dx \dt))^d$.
    \item[ii)] The triple $(\rho_\eps, \varphi_\eps, \mathbf{u}_\eps)$  satisfies \begin{equation}\label{eq:transport}
        - \int_0^T \int_{\Omega} \rho_\eps (\partial_t \eta + \mathbf{u}_\eps  \cdot \nabla \eta ) \, \dx \dt = \int_{\Omega} \rho_0 \eta(\cdot,0) \, \dx
    \end{equation}
    for all $\eta \in C^\infty_c(\overline\Omega \times [0, T))$, 
    \begin{align}\label{eq:weaksol}
&- \int_0^T \int_\Omega \rho_\eps  \mathbf{u}_\eps \cdot \xi  \, \dx \dt \\
&= 
\int_0^T \int_{\Omega}  \bigg( \Big(\frac{\eps}{2} |\nabla \varphi_\eps|^2 + \frac{1}{\eps} F(\varphi_\eps)  \Big) \operatorname{Id}  - \eps \nabla \varphi_\eps\otimes  \nabla \varphi_\eps \bigg):\nabla \xi  
   \, \mathrm{d}x \dt  \notag \\
 & \quad -\int_0^T \int_{\Omega}  \mathrm{div} \xi\, \mathrm{d}\mu_\eps +\int_0^T \int_{\Omega} \lambda\varphi_\eps \mathrm{div} \xi \,  \dx \dt,  \notag
    \end{align}
    for all $\xi \in C^\infty_c(\overline \Omega \times [0, T); \mathbb{R}^d) $ with $\xi \cdot \nu_{\Omega} =0$ along $\partial \Omega \times (0,T)$, 
  where $\mu_\eps \in \mathcal{M}_b(\Omega \times (0,T))$ is such that $\mu_\eps \geq \rho_\eps^m \mathcal{L}^d \llcorner \Omega \otimes \mathcal{L}^1 \llcorner (0,T)$,
    and 
     \begin{equation}\label{eq:ACweak}
    \int_0^T \int_{\Omega} \bigg ( \eps \nabla \varphi_{\epsilon} \cdot \nabla \eta +  \Big (  \frac1\eps F^\prime(\varphi_{\epsilon}) + \lambda-\rho_{\epsilon} \Big ) \eta \bigg ) \, \dx \dt= {0},
\end{equation}
for all $\eta \in C^\infty_c(\overline\Omega \times [0, T))$.

    \item[iii)]  For almost every $T' \in [0,T]$, the energy dissipation rate 
    \begin{align}\label{eq:energydiss}
    \mathcal{E}_\eps(\rho_\eps(T'), \varphi_\eps(T')) + \int_0^{T'} \int_\Omega \rho_\eps |\mathbf{u}_\eps |^2\, \dx \dt \leq \mathcal{E}_\eps(\rho_0, \varphi_{0}) . 
    \end{align}
    holds true.
\end{itemize}
\end{definition}

We now state our main existence result for weak solutions to the phase-field model. The proof relies on the variational structure of~\eqref{eq:phasefield} and uses a time-discretization scheme inspired by~\cite{Schaetzle}.
\begin{theorem}[Existence of weak solutions]\label{theo:existence}
       Let $d \in \{2,3\}$ and $m \geq 2$.
       Let  $\rho_0 \in L^m(\Omega;[0, \infty))$ be  such that $\int_\Omega \rho_0 \, \dx =1 $ and let $\varphi_0 \in H^1(\Omega)$ be a minimizer of the functional $\varphi \mapsto \mathcal{E}_{\textnormal{int}, \eps}(\varphi) +\int_{\Omega} \varphi (\lambda-\rho_0) \, \dx $.
       Then there exists a weak solution $(\rho_\eps, \varphi_\eps, \mathbf{u}_\eps)$ to the phase-field model~\eqref{eq:phasefield}, in the sense of Definition~\ref{def:weaksol}, which additionally satisfies the almost-minimizing property
\begin{align} 
    & \int_0^T\int_{{\Omega}} \Big ( \frac{\varepsilon}{2}|\nabla \varphi_\eps|^2 + \frac{1}{\varepsilon} F(\varphi_\eps) \Big ) \, \dx \dt +\int_0^T \int_{{\Omega}} \varphi_\eps  (\lambda-\rho_\eps)  \, \dx \dt \notag \\
    &\le \int_0^T\int_{{\Omega}} \Big ( \frac{\varepsilon}{2}|\nabla \tilde \varphi|^2 + \frac{1}{\varepsilon} F(\tilde \varphi) \Big ) \, \dx \dt
    +  \int_0^T\int_{{\Omega}} \tilde \varphi (\lambda- \rho_\eps)  \, \dx\dt + \eps
    \label{eq:almostmin}
\end{align}
for every $\tilde \varphi \in L^1(0,T;H^1({\Omega}))$.
\end{theorem}

\begin{remark} 
    Notice that, if we assumed the convergence of the total energy for our minimization scheme~\eqref{eq:ACopt}-\eqref{minpb}, namely
    \begin{align*}
        \int_0^T\mathcal{E}_\eps(\rho_\eps^h,\varphi_\eps^h) \, \dt \to \int_0^T\mathcal{E}_\eps(\rho_\eps,\varphi_\eps) \, \dt\quad \text{as } h \to 0,
    \end{align*}
    or, alternatively, if one could prove that 
    \begin{align*}
          \int_0^T \int_\Omega (\rho_\eps^h)^m  \, \dx \dt \to \int_0^T \int_\Omega (\rho_\eps)^m  \, \dx  \dt \quad \text{as } h \to 0,
    \end{align*}
    then we could identify the defect measure $\mu_\eps$ with $\rho_\eps^m \mathcal{L}^d \llcorner \Omega \otimes \mathcal{L}^1 \llcorner (0,T)$. Indeed, we will establish the convergence of the diffuse interfacial energy and of the coupling term as $h \to 0$ (cf. Proposition~\ref{prop:sel-phase-compactness}), so that the only possible loss in the total energy is due to the density term.
\end{remark}

\begin{remark}
    We observe that any sufficiently smooth triple $(\rho_\eps, \varphi_\eps,\mathbf{u}_\eps)$ satisfying~\eqref{eq:transport}-\eqref{eq:ACweak} with $\mu_\eps = \rho_\eps^m \mathcal{L}^d \llcorner \Omega \otimes \mathcal{L}^1 \llcorner (0,T)$, whose existence is established by the theorem above, is in fact a strong solution to the phase-field model~\eqref{eq:phasefield}. This can be shown via integrations by parts.
\end{remark}

\subsection{Minimization problem and discrete flow}\label{sec:minprob}
Let $\rho_0 \in L^m(\Omega;[0,\infty))$ be an initial datum and let $\varphi_0 \in H^1(\Omega)$ be a minimizer of the functional $\varphi \mapsto \mathcal{E}_{\textnormal{int}, \eps}(\varphi) +\int_{\Omega} \varphi (\lambda-\rho_0) \, \dx $.
Fix $h>0$, our implicit time discretization argument reads as follows.  Set $(\rho_{\varepsilon,0}^h,\varphi_{\varepsilon,0}^h):= (\rho_0, \varphi_0)$, and, for $n \in \mathbb{N} \setminus \{0\}$,  let $\varphi_{\varepsilon, n}^h$ be a solution of
\begin{equation}\label{eq:ACopt}
   -  \eps\Delta \varphi_{\varepsilon, n}^h + \frac{1}{\eps} F^\prime(\varphi_{\varepsilon, n}^h) + \lambda-\rho_{\varepsilon, n-1}^h= {0}
   \quad  \text{ distributionally in } \Omega,
\end{equation}
endowed with Neumann boundary condition, satisfying
\begin{equation}\label{eq:ModicaMortoladecr}
   \mathcal{E}_{\textnormal{int}, \eps}(\varphi_{\varepsilon, n}^h) +\int_{\Omega} \varphi_{\varepsilon, n}^h (\lambda-\rho_{\varepsilon, n-1}^h) \, \dx \le \mathcal{E}_{\textnormal{int}, \eps}(\varphi_{\varepsilon, n-1}^h) +\int_{\Omega} \varphi_{\varepsilon, n-1}^h (\lambda-\rho_{\varepsilon, n-1}^h) \, \dx,
\end{equation}
and let $\rho_{\varepsilon, n}^h$ be such that
\begin{equation}\label{minpb}
\rho_{\varepsilon,n}^h \in \arg \min	\Big\{  \frac{1}{2h} W^2_2(\rho, \rho^h_{\varepsilon,n-1}) +     \mathcal{E}_\eps(\rho, \varphi_{\varepsilon,n}^h) : \rho \in \mathcal{K}\Big\},
\end{equation}
where
\[
\mathcal{K}:= \Big \{\rho \in L^1(\Omega): \rho \ge 0, \, \int_{\Omega} \rho \, \dx=1 \Big \}.
\]
We call $(\rho_{\varepsilon,n}^h,\varphi_{\varepsilon,n}^h)_{n \in \mathbb{N}}$ a discrete solution starting from $(\rho_0, \varphi_0)$. Moreover, we define the approximate solution $(\rho^h_{\epsilon}(t,\cdot), \varphi^h_{\epsilon}(t,\cdot))_{t \in [0,\infty)}$ by piecewise constant interpolation, namely
\begin{equation*}
\rho^h_{\epsilon}(t,\cdot):=\rho_{\epsilon,n}^h(\cdot), \quad \varphi^h_{\epsilon}(t,\cdot):=\varphi^h_{\epsilon,n}(\cdot)\quad \text{for every } t \in [nh,(n+1)h).
\end{equation*}

By testing~\eqref{minpb} with $\rho=\rho_{\varepsilon,n-1}^h$ and using~\eqref{eq:ModicaMortoladecr}, we obtain
\begin{equation}\label{eq:energydecr}
    \frac{1}{2h} W_2^2(\rho_{\varepsilon,n}^h,\rho_{\varepsilon,n-1}^h) + \mathcal{E}_\varepsilon(\rho_{\varepsilon,n}^h,\varphi_{\varepsilon,n}^h) \le \mathcal{E}_\varepsilon(\rho_{\varepsilon,n-1}^h,\varphi_{\varepsilon,n-1}^h).
\end{equation}
In particular, the energy $\mathcal{E}_\varepsilon$ is not increasing along the discrete scheme.

\begin{remark}\label{rk:existence}
   We observe that, under the assumption $m \ge 2$, the scheme~\eqref{eq:ACopt}-\eqref{minpb} is well defined. Indeed, the sequence $(\rho^h_{\epsilon,n}, \varphi^h_{\epsilon,n})_{n \in \mathbb{N}}$, defined recursively by   
    \begin{equation*}\label{eq:minvarphi}
   \varphi^h_{\varepsilon,n} \in \arg\min_{\varphi \in H^1(\Omega)} \mathcal{E}_{\textnormal{int},\varepsilon}(\varphi)+\int_{\Omega} \varphi (\lambda-\rho_{\varepsilon,n-1}^h) \, \dx
   \end{equation*}
   and $\rho_{\eps,n}^h$  given by~\eqref{minpb}, satisfies~\eqref{eq:ACopt}-\eqref{minpb}.

   To show existence of such sequence
   for {$m \ge 2$}, we observe that the functional $\mathcal{E}_\eps(\rho,\varphi)$ defined in~\eqref{phasefieldenergy} is bounded from below, coercive and all of its terms are lower semicontinuous. Moreover, at each step $\mathcal{E}_\eps(\rho,\varphi)$ is non-increasing.
   Using Fenchel's inequality with $f(s)=\frac{1}{m-1}s^m$ and $f^*(s) \simeq s^\frac{m}{m-1}$, we obtain
     \[
   \int_{\Omega} \varphi\rho  \, \dx \le \int_{\Omega} \frac{1}{2} \big ( f(\rho)+f^*(\varphi) \big ) \, \dx,
   \]
yielding
   \[
   \frac{\varepsilon}{2} \| \nabla \varphi\|_{L^2(\Omega)}^2 + \frac{1}{\varepsilon}\int_\Omega F(\varphi) \, \dx +\frac{1}{2} \int_\Omega \frac{\rho^m}{m-1} \, \dx - C_m \| \varphi \|^{\frac{m}{m-1}}_{L^{\frac{m}{m-1}}(\Omega)}+ \lambda \int_\Omega \varphi \, \dx \le \mathcal{E}_\eps(\rho,\varphi).
   \]
   The functional $\mathcal{E}_\eps(\rho,\varphi)$ is bounded from below whenever $m > 4/3$, since in this case $1<m/(m-1)<4$ and we can absorb the lower-order terms  into $\frac{1}{\varepsilon}F(\varphi)$. Moreover, for $\eps>0$ small enough, 
   $\mathcal{E}_\eps(\rho,\varphi)$ is bounded also in the case $m = 4/3$. Instead, if we were considering larger values of $\eps>0$, we would have that $\mathcal{E}_\eps(\rho,\varphi)$ is bounded from below whenever $m > 1$ in the case $d=2$, due to an application of the Sobolev inequality. 
   In particular, since for $m=1$ we  consider $f(s)=s \log s$ and in this case $f^*(s) = \exp(s-1)$,
   we cannot deduce the boundedness from below of $\mathcal{E}_\eps(\rho,\varphi)$ for any fixed $\eps>0$, unless we consider a potential $F$ of exponential growth. More in general, we observe that one could consider smaller values of $m$ by assuming higher polynomial growth of the potential.
   In particular, it follows that the functional $\mathcal{E}_\varepsilon(\rho,\varphi)$ is coercive and it holds,  for $\varepsilon>0$ sufficiently small,
   \begin{equation}\label{eq:coercivity}
       \varepsilon\|\nabla \varphi\|_{L^2(\Omega)}^2+\frac{1}{\varepsilon}\int_\Omega F(\varphi) \, \dx+\int_\Omega \frac{\rho^m}{m-1} \, \dx \le  C\mathcal{E}_\varepsilon(\rho,\varphi)+C,
   \end{equation}
   where $C=C(m,\lambda,|\Omega|)>0$.
   Moreover, for $m\ge 2$, one can show that all the terms of $(\rho,\varphi)\mapsto \mathcal{E}_\varepsilon(\rho,\varphi)$ are
   lower semicontinuous with respect to weak $L^m$-convergence in $\rho$ and weak $H^1$-convergence in $\varphi$.
\end{remark}

\begin{remark}\label{rmk:enpos}
For convenience, since $\mathcal{E}_\eps(\rho,\varphi)$ is bounded from below, we may shift the energy~\eqref{phasefieldenergy} by a suitable constant so that, for every $\varphi$,
$
\inf_{\rho\in\mathcal K}\mathcal{E}_\eps(\rho,\varphi)\geq 0.
$
Clearly, the addition of the constant does not affect the dynamics.
\end{remark}

Using the stationary Allen-Cahn equation~\eqref{eq:ACopt}, we obtain improved regularity for the phase field $\varphi_\varepsilon$ in the following lemma.
\begin{lemma}[Regularity of $\varphi_n^h$]\label{lemma:regularity}
     For any $h>0$ and every $n=1,2,\ldots$, 
     it holds that $\varphi_n^h \in H^2(\Omega)$ and $\nabla \varphi_n^h$ is uniformly bounded in $L^{q}(\Omega)$, for some $q > 2$, by a constant only depending on $\|\varphi_n^h\|_{H^1(\Omega)}$ and $\|\rho^h_{n-1}\|_{L^m(\Omega)}$. 
\end{lemma}
\begin{remark}\label{remarkd} 
Notice that the restrictions $m\geq 2$ and $d\in\{2,3\}$ in the hypotheses of Theorem~\ref{theo:existence} are used to obtain the regularity $\varphi_n^h\in H^2(\Omega)$.  More precisely, $m\geq 2$ ensures that the density $\rho_n^h$ belongs to $L^2(\Omega)$, while the restriction $d\leq 3$ allows us to control the term $F'(\varphi_n^h)$ in $L^2(\Omega)$. 
\end{remark}
\begin{proof}
    Recall that $\rho_{n-1}^h \in L^m(\Omega)$ with $m \ge 2$ and $\varphi_{n}^h \in L^6(\Omega)$ by the embedding $H^1 \hookrightarrow L^6$ which holds if $d=2,\, 3$, in particular $F'(\varphi_{n}^h) \in L^{2}(\Omega)$.
    By elliptic regularity, it follows from~\eqref{eq:ACopt} that $\varphi_{n}^h\in H^2(\Omega) $. 
     In particular, $\nabla \varphi_{n}^h\in H^1(\Omega) $. On the other hand, we have that $H^1(\Omega) \hookrightarrow L^q(\Omega)$, for every $q\ge 1$ if $d=2$ and $q=6$ if $d=3$. 
\end{proof}

Since $\eps>0$ is fixed throughout this section, we will omit the index $\eps$ from the notation in what follows.

Computing the variations of $\mathcal{E}$ in  $\rho^h_{n}$, one obtains the following optimality conditions (see~\cite[Proposition 8.7]{Santabook}).

\begin{lemma}[Euler-Lagrange equations]\label{phivarphi}
    For any $h>0$ and every $n \ge 1$, there exists a 
    constant $c_{n}^h\in \mathbb{R}$ such that the optimal pair $(\rho^h_{n},\varphi_{n}^h)$ satisfies
        \begin{align}\label{eq:santaopt}
         \tfrac{m}{m-1}(\rho^h_{n})^{m-1} - \varphi_{n}^h + \phi_{n}^h&= c_{n}^h  \quad \text{ a.e. in } \supp(\rho_{n}^h), 
        \\
      \tfrac{m}{m-1}\nabla (\rho^h_{n})^{m-1} - \nabla \varphi_{n}^h +\nabla \phi_{n}^h&= 0 \quad \text{ a.e. in } \supp(\rho_{n}^h), \label{eq:santaoptgrad}
    \end{align}
    where $\phi^h_{n}$ is the  Kantorovich potential from $\rho^h_{n}$ to $\rho_{n-1}^h$, i.e. $(\operatorname{Id}-h \nabla \phi^h_{n})_\# \rho^h_{n}=\rho^h_{n-1}.$ 
\end{lemma}

Next we derive the equations satisfied by the discrete flow.
\begin{lemma}[Discrete equation for the velocity field]\label{lemma:EL}
     Given~$h>0$, $\rho_0 \in L^m(\Omega)$ such that $\int_\Omega \rho_0 \,\dx=1$ and $\varphi_0$ a minimizer of the functional $\varphi \mapsto \mathcal{E}_{\textnormal{int}, \eps}(\varphi) +\int_{\Omega} \varphi (\lambda-\rho_0) \,\dx $, let $(\rho_n^h,\varphi^h_n)_{n \in \mathbb{N}}$  be a discrete solution starting from $(\rho_0, \varphi_0)$.
   Then, for any vector field $\xi\in C^1_c(\overline\Omega;\mathbb{R}^d)$ with $\xi \cdot \nu_\Omega =0$ on $\partial \Omega$, we have 
    \begin{align}\label{eq:eulerlagrdis}
   &- \int_{\Omega} \rho^h_{n} \nabla \phi^h_n \cdot \xi \, \mathrm{d} x  \notag  \\
    &= \int_{\Omega} \Big(\frac\eps2  |\nabla \varphi^h_{n}|^2 
+ \frac1\eps F( \varphi^h_{n} ) \Big) \div \xi 
- \eps \nabla \varphi^h_{n} \otimes \nabla \varphi^h_{n} : \nabla \xi \, \dx \\
& \quad - \int_{\Omega } (\rho^h_{n})^m \operatorname{div}\xi \, \mathrm{d} x
  +\lambda \int_{\Omega } \varphi^h_{n}\operatorname{div} \xi \, \mathrm{d} x + R_n^h(\xi) , \notag
\end{align}
where $R_n^h(\xi) :=- \int_{\Omega} (\rho^h_{n} - \rho^h_{n-1})\nabla \varphi^h_n \cdot \xi \, \mathrm{d} x$.
\end{lemma}

\begin{proof}[Proof of Lemma~\ref{lemma:EL}]
Let $\xi\in C^1_c(\overline\Omega;\mathbb{R}^d)$ with $\xi \cdot \nu_\Omega =0$ on $\partial \Omega$.
Multiplying~\eqref{eq:santaoptgrad} by $\rho_n^h$ and then testing by $\xi$, we obtain
\begin{equation*}
   \int_\Omega  \frac{m}{m-1} \rho_n^h\nabla (\rho^h_{n})^{m-1}  \cdot  \xi\, \dx- \int_{\Omega} \rho_n^h\nabla \varphi_{n}^h \cdot  \xi \, \dx+\int_{\Omega}\rho_n^h \nabla \phi_{n}^h \cdot \xi \, \dx= 0,
\end{equation*}
from which we deduce, by integrating by parts,
\begin{equation*}
   -\int_{\Omega}\rho_n^h \nabla \phi_{n}^h \cdot \xi \, \dx= -\int_\Omega (\rho^h_{n})^{m}  \operatorname{div} \xi\, \dx- \int_{\Omega} \rho_n^h\nabla \varphi_{n}^h \cdot \xi \, \dx.
\end{equation*}
By Lemma~\ref{lemma:regularity} we can use $\nabla \varphi_{n}^h \cdot \xi \in H^1(\Omega)$ as test function in~\eqref{eq:ACopt} and obtain
\begin{equation*}
      \int_\Omega \rho_{ n-1}^h\nabla \varphi_{n}^h \cdot \xi \, \dx= \int_{\Omega} \eps \nabla \varphi_{ n}^h \cdot \nabla(\nabla \varphi_{n}^h \cdot \xi)  + \frac{1}{\eps} F^\prime(\varphi_{ n}^h) \nabla \varphi_{n}^h \cdot \xi \, \dx  + \int_\Omega \lambda \nabla \varphi_{n}^h \cdot \xi \, \dx.
\end{equation*}
Integrating by parts once again, we deduce
\begin{align*}
      &-\int_\Omega \rho_{n-1}^h\nabla \varphi_{n}^h \cdot \xi \, \dx\\
      &= \int_{\Omega}-\eps \nabla \varphi_{ n}^h \cdot \nabla(\nabla \varphi_{n}^h \cdot \xi)  - \frac{1}{\eps}  F^\prime(\varphi_{ n}^h) \nabla \varphi_{n}^h \cdot \xi \, \dx  + \int_\Omega \lambda  \varphi_{n}^h \operatorname{div} \xi \, \dx\\
      &= \int_{\Omega} \eps \nabla \varphi^h_{n} \cdot \Big( - \nabla \xi \nabla   \varphi^h_{n}  
    - \nabla^2\varphi^h_{n} \xi \Big) 
    - \frac1\eps F'( \varphi^h_{n} ) \nabla \varphi^h_{n} \cdot  \xi \, \dx + \int_\Omega \lambda  \varphi_{n}^h \operatorname{div} \xi \, \dx \\
    &= \int_{\Omega} \Big(\frac\eps2  |\nabla \varphi^h_{n} |^2 
    + \frac1\eps F( \varphi^h_{n} ) \Big) \div \xi 
    - \eps \nabla \varphi^h_{n} \otimes \nabla \varphi^h_{n} : \nabla \xi \, \dx + \int_\Omega \lambda  \varphi_{n}^h \operatorname{div} \xi \, \dx. 
\end{align*}
\end{proof}

\subsection{Compactness}\label{sec:compactness}
In this section, we establish suitable convergence results for $(\rho^h,\varphi^h)_{h >0}$, which will allow us to pass to the limit $h \to 0$.

\subsubsection{Weak compactness}
For the proof of the following result on the weak convergence of $\rho^h$ we refer the reader to~\cite[Lemma 6]{KLM}.
\begin{lemma}
\label{lemma:compactness rho}
    Given~$h>0$, $\rho_0 \in L^m(\Omega)$ such that $\int_\Omega \rho_0 \,\dx=1$ and $\varphi_0$ a minimizer of the functional $\varphi \mapsto \mathcal{E}_{\textnormal{int}, \eps}(\varphi) +\int_{\Omega} \varphi (\lambda-\rho_0) \, \dx $, let $(\rho_n^h,\varphi^h_n)_{n \in \mathbb{N}}$  be a discrete solution starting from $(\rho_0, \varphi_0)$. Then
\begin{equation} \label{eq:discomp}
    W_2(\rho^h(t), \rho^h(s)) \leq \sqrt{2}\mathcal{E}(\rho_0, \varphi_0)^\frac12 (t-s)^\frac12,
\end{equation}
for $t>s \ge0$ with $t-s\ge h$.
Furthermore, for every $T>0$, we have up to a subsequence
\begin{equation}\label{eq:weakcomprho}
\rho^h \rightharpoonup \rho \quad \text{ weakly in } L^{m}(\Omega \times (0,T)),
\end{equation}
there exists $\mu \in \mathcal{M}_b(\Omega \times (0,T))$ such that $(\rho^h)^m \overset{*}{\rightharpoonup} \mu$ weakly-* as measures,
and
\begin{equation}\label{eq:contcomp}
    W_2(\rho(t), \rho(s)) \leq \sqrt{2}\mathcal{E}(\rho_0, \varphi_{0})^\frac12 (t-s)^\frac12,
\end{equation}
for $t>s\ge 0$.
\end{lemma}

The weak convergence properties of $\varphi^h$ follow directly from the energy bound and Lemma~\ref{lemma:regularity} and can be collected in the following statement.
\begin{lemma}
\label{lemma:compactness chi}
       Given~$T>0$,  $\rho_0 \in L^m(\Omega)$ such that $\int_\Omega \rho_0 \,\dx=1$ and $\varphi_0$ a minimizer of the functional $\varphi \mapsto \mathcal{E}_{\textnormal{int}, \eps}(\varphi) +\int_{\Omega} \varphi (\lambda-\rho_0) \, \dx $, let $(\rho_n^h,\varphi^h_n)_{n \in \mathbb{N}}$  be a discrete solution starting from $(\rho_0, \varphi_0)$.
      Then, up to subsequences,
      \begin{equation}\label{eq:weakcompvarphi}
\varphi^h \rightharpoonup \varphi \quad \text{ weakly in } L^4(\Omega \times (0,T)),
\end{equation}
and
\begin{equation}\label{eq:weaknablavarphi}
\nabla \varphi^h \rightharpoonup \nabla \varphi \quad \text{ weakly in } L^q(\Omega \times (0,T)), 
\end{equation}
where the exponent $q$ is given by Lemma~\ref{lemma:regularity}.
\end{lemma}
\begin{proof}
    Let $N=[T/h]+1$. Using~\eqref{eq:energydecr} and the bound~\eqref{eq:coercivity}, we obtain
     \begin{align*}
    \| \varphi^h\|^4_{L^4(\Omega \times (0,T))} &\le h\sum_{n=0}^{N}  \int_\Omega (\varphi_n^h)^4 \, \dx \\
    & \leq  C h\sum_{n=0}^{N}  \Big(\int_\Omega F(\varphi_n^h) \,  \dx + C \Big) 
    \\
    & \leq  C h\sum_{n=0}^{N} \Big( \mathcal{E}(\rho^h_{n}, \varphi_{n}^h)  +C\Big)  \leq C T \Big( \mathcal{E}(\rho_{0}, \varphi_{0})  +C\Big ),
    \end{align*}
    whence the weak convergence~\eqref{eq:weakcompvarphi} follows.
   Moreover, using again~\eqref{eq:energydecr} and~\eqref{eq:coercivity} combined with the uniform bound from Lemma~\ref{lemma:regularity}, we have
    \begin{align*}
    \| \nabla \varphi^h\|^q_{L^q(\Omega \times (0,T))} = h\sum_{n=0}^{N} \int_\Omega (\nabla \varphi_n^h)^q \, \dx 
    \leq C T, 
    \end{align*}
    yielding~\eqref{eq:weaknablavarphi}.
\end{proof}


\subsubsection{Strong compactness of $\rho_n^h$}
By Kantorovich duality we have
\begin{align}\label{eq:kanto}
\frac{1}{2h} W^2_2(\rho_n^h, \rho_{n-1}^h) = \frac{h}{2} \int_{\Omega} |\nabla \phi_n^h|^2 \rho_n^h \, \text{d}x,
\end{align}
where we recall that $\phi_n^h$ is the Kantorovich potential. 
Moreover, we have that
\begin{equation*}
\begin{split}
    \frac{h}{2} \sum_{n=0}^{N} \int_{\Omega} |\nabla \phi_n^h|^2 \rho_n^h \, \text{d}x
    \le \mathcal{E}(\rho_{0},\varphi_{0})-\mathcal{E}(\rho_{N}^h,\varphi_{N}^h),
\end{split}
\end{equation*}
whence, letting $T=Nh$,
\begin{equation}\label{phibound}
    \frac{1}{2}\int_0^T \int_{\Omega} |\nabla \phi^h|^2 \rho^h \, \text{d}x \text{d}t \le \mathcal{E}(\rho_{0},\varphi_0).
\end{equation}

Similarly as in~\cite[Proposition 8]{KLM} we obtain the following strong convergence.
\begin{proposition}\label{prop:strongrho}
    Given~$T>0$, $\rho_0 \in L^m(\Omega)$ such that $\int_\Omega \rho_0 \,\dx=1$ and $\varphi_0$ a minimizer of the functional $\varphi \mapsto \mathcal{E}_{\textnormal{int}, \eps}(\varphi) +\int_{\Omega} \varphi (\lambda-\rho_0) \, \dx $, let $(\rho_n^h,\varphi^h_n)_{n \in \mathbb{N}}$  be a discrete solution starting from $(\rho_0, \varphi_0)$. Then up to a subsequence
    \begin{align} \label{eq:strongcomprhop}
    &\rho^h \rightarrow \rho \quad \text{ strongly in } L^p(\Omega \times (0,T)), \quad \text{ for } 1\le p<m.
    \end{align}
\end{proposition}
\begin{proof} 
We define an auxiliary function $g: (0, + \infty] \rightarrow (0, + \infty]$ as
\begin{align*}
    g(z) := \begin{cases}
        z^{m}, & z<1,\\
        z^{m-1} , & z\ge 1,\\
        + \infty, & z=+\infty,
    \end{cases}
\end{align*}
which is continuous and strictly increasing, hence invertible.

Consider $g\circ \rho^h_n$, in view of~\eqref{eq:energydecr} and~\eqref{eq:coercivity}, we deduce that
\[
\int_\Omega {|g (\rho^h_n)|^s} \, \dx \leq C\mathcal{E}(\rho_{0},\chi_{E_{0}})+C,
\]
where $C>0$ is a uniform constant which may vary from line to line and \[s \in \Big (1, \min \Big\{\frac{m}{m-1}, \frac{2m}{m+1} \Big\} \Big ].\]
 To estimate the $L^s$-norm of $\nabla (g \circ \rho_n^h)$, we compute
\begin{align*}
   \int_\Omega |\nabla (g \circ  \rho_n^h) |^s \, \dx
    &\leq  C \int_{\Omega} {(\rho_n^h)}^{\frac{s}{2}}|\nabla   f'(\rho_n^h) |^s\, \dx\\
    &\le C \int_{\Omega} ({\rho_n^h})^{\frac{s}{2}}|\nabla   \varphi_n^h |^s\, \dx+C \int_{\Omega} ({\rho_n^h})^{\frac{s}{2}}|\nabla   \phi_n^h |^s\, \dx,
\end{align*}
where we used~\eqref{eq:santaoptgrad}. 
Notice that the second term on the right-hand side is bounded due to~\eqref{phibound}, whereas the first one can be bounded using Young's inequality with exponents $2m/s$ and $2m/(2m-s)$ together with~\eqref{eq:coercivity}.
Therefore, we obtain 
\begin{equation}
    \int_0^T \|(g \circ \rho^h)(\cdot,t)\|_{W^{1,s}(\Omega)}\, \dt  \le C.
\end{equation}
We introduce the lower semicontinuous functional $\mathcal{F}: L^s(\Omega) \to [0,+\infty]$ defined as
        \[
    \mathcal{F}(v) := \int_\Omega \frac{v^m}{m-1} \, \dx + \|g(v)\|_{W^{1,s}(\Omega)}.
    \]
We observe that, for every $C>0$, the set $\mathcal{V}:=\{
v \in L^s(\Omega): \mathcal{F}(v) \le C\} 
$
is compactly embedded in $L^s(\Omega)$ under the strong topology. To prove this, consider a sequence $(v_n)_n \subset \mathcal{V}$ and notice that, up to a subsequence, $g(v_n)$ converges to some $ w$ strongly in $ L^s(\Omega)$ and almost everywhere in $\Omega$. The strict monotonicity of $g$ guarantees its invertibility, implying
$
v_n = g^{-1}(g(v_n)) \to g^{-1}(w)$ a.e. in $\Omega$.
 The strong convergence in $L^s(\Omega)$ then follows from the dominated convergence theorem.
Combining this with~\eqref{eq:discomp}, all the assumptions of Theorem~\ref{theo:RossiSava} are satisfied with $d=W_2$ (Wasserstein distance) and $\mathcal{U}=(\rho^h)_{h>0}$, in particular $\mathcal{U}$ is tight with respect to $\mathcal{F}$. Applying Theorem~\ref{theo:RossiSava}, we obtain that $\mathcal{U}$ is relatively compact in $\mathcal{M}(0,T; L^s(\Omega))$, yielding the strong convergence $\rho^h(\cdot,t) \rightarrow \rho(\cdot,t) $ in $L^s(\Omega)$ for almost all $t \in (0,T)$. Finally, the strong convergence in $L^s(\Omega \times (0,T))$ is a direct consequence of the dominated convergence theorem.
Furthermore, we obtain~\eqref{eq:strongcomprhop} via an interpolation argument. 
\end{proof}

Defining the affine interpolation by
\[
    \hat\rho^h(t)
    :=
    \frac{nh-t}{h}\,\rho_{n-1}^h
    +
    \frac{t-(n-1)h}{h}\,\rho_n^h,
    \qquad
    t\in[(n-1)h,nh],
\]
by the Arzel\'a-Ascoli theorem, we also have
\begin{align}\label{eq:convrhoH-1}
    \hat\rho^h \to \rho
    \qquad\text{strongly in }
    C^0([0,T];(H^1(\Omega))').
\end{align}

\subsubsection{Strong compactness of $\varphi^h_n$ and almost minimality}\label{sec:strongcompvarphi}
Since the phase equilibrium problem~\eqref{eq:ACopt} may admit several solutions, arbitrary choices at successive time steps can lead to oscillatory phenomena between the wells and prevent strong compactness of $\varphi^h_n$. 
Therefore, at each time step, we choose $\varphi^h_n$ by Sch\"{a}tzle's selection procedure (cf.~\cite[Section~2]{Schaetzle}). 
The selection argument consists of five main steps, which we summarize below.
\begin{itemize}
    \item[\textit{(i)}] Starting from a
global minimizer $(\bar\rho,\bar\varphi)$ of the total energy~\eqref{phasefieldenergy}, we describe the local structure of the equilibria of the functional
\begin{equation}\label{eq:Frho}
    \varphi\mapsto
    \mathcal F_\rho(\varphi)
    :=
    \mathcal E_{\operatorname{int},\eps}(\varphi)
    +\int_\Omega \varphi(\lambda-\rho)\,\dx .
\end{equation}
in a neighbourhood of
$(\bar\rho,\bar\varphi)$. We reduce the
equilibrium Allen--Cahn equation with forcing $\rho$ to a
one-dimensional critical-point condition,
which provides the local properties needed
for the selection procedure. 
\item[\textit{(ii)}] The local properties from \textit{(i)} are used to
define the selected phase $\varphi_n^h$ at each discrete step. 
As long as the density remains in the current local neighbourhood, the selection procedure can be continued, 
guaranteeing both
\[
   -\varepsilon \Delta \varphi_n^h +\frac{1}{\varepsilon}F'(\varphi_n^h)+\lambda =\rho_{n-1}^h
\]
and the phase-energy decay
\[
    \mathcal F_{\rho_{n-1}^h}(\varphi_n^h)
    \leq
    \mathcal F_{\rho_{n-1}^h}(\varphi_{n-1}^h).
\]
If the density leaves the current neighbourhood, the procedure is restarted
from a global minimizer of $\mathcal F_{\rho_{n-1}^h}$.

\item[\textit{(iii)}] We show that the number of restarts on $[0,T]$ is bounded uniformly with respect to $h$. 
Consequently, after passing to a subsequence, the restart times converge to finitely many limiting times.

\item[{\textit{(iv)}}] We work between two consecutive limiting restart
times and exploit the one-dimensional structure given by
{\textit{(i)}}. A monotonicity argument for the indices of the
selected critical points, together with the properness of $\mathcal A$, rules out oscillations between different equilibria and yields the strong
convergence~\eqref{eq:strongconvvarphi} (cf. Proposition~\ref{prop:sel-phase-compactness} below).
 
\item[{\textit{(v)}}] The same local construction also yields the phase almost-minimizing property~\eqref{eq:almostmin}.
\end{itemize}

We now present each step in full detail.

\smallskip
\paragraph{\textit{i) Local structure of equilibria.}}
For  $\rho\in L^2(\Omega)$, we consider the functional~\eqref{eq:Frho}.
It follows from~\cite[Theorem~2.4 and Remark~2.6]{CanariusSchaetzle}
that, for every fixed $\rho\in L^2(\Omega)$, the functional
$\mathcal F_\rho$ has only finitely many global minimizers in
$H^1(\Omega)$. In fact, the set of critical points of
$\mathcal F_\rho$ with nonnegative second variation is finite.

We introduce the operator $\mathcal{A}:H^1(\Omega) \to (H^{1}(\Omega))'$ defined by
\begin{align*}
    \mathcal{A}(\varphi) := -\varepsilon \Delta \varphi +\frac{1}{\varepsilon}F'(\varphi)+\lambda
\end{align*}
with homogeneous Neumann boundary condition. The operator $\mathcal{A}$ is proper (see~\cite[Remark 2.2]{CanariusSchaetzle}), i.e., if $\mathcal{A}(\varphi_j) \to g$ strongly in $(H^1(\Omega))'$ as $j \to \infty$,
then up to subsequences $\varphi_j \to \varphi$ strongly in $H^1(\Omega)$ and $\mathcal{A}(\varphi)=g $.
Notice that 
\begin{equation} \label{eq:nullvar}
    \delta \mathcal{F}_{\rho}(\varphi) =0 \iff \mathcal{A}(\varphi) = \rho,
\end{equation}
corresponding to~\eqref{eq:ACopt}.

The bilinear form associated with the second variation of
$\mathcal F_\rho$ at $\varphi$ is
\[
    B_\varphi(v,w)
    :=
    \int_\Omega
    \varepsilon\nabla v\cdot\nabla w
    +\frac1\varepsilon F''(\varphi)vw\,\dx .
\]
After identifying $(H^1(\Omega))'$ with $H^1(\Omega)$ by means of
the Riesz isomorphism, $B_\varphi$ induces a self-adjoint Fredholm
operator on $H^1(\Omega)$.

Let now $\bar\varphi$ be a global minimizer of
$\mathcal F_{\bar\rho}$, for some fixed $\bar\rho\in L^2(\Omega)$.
Then $B_{\bar\varphi}(v,v)\ge0 $ for every $v\in H^1(\Omega)$. Therefore, all eigenvalues are real and nonnegative.
Moreover, from standard elliptic theory
(cf.~\cite[Section 8.12]{GilbargTrudinger}), it follows that the smallest eigenvalue is simple and has a positive eigenfunction $\bar v$ with $\|\bar v\|_{H^1(\Omega)}=1$. We decompose $H^1(\Omega) = Y \oplus \operatorname{span}(\bar v)$, where $Y:= \bar v^\perp$. Notice that the restriction $B_{\bar\varphi} |_{Y}$ is coercive.


We define $\Phi:\mathbb{R} \times Y \times Y \to Y$ by
\[
\Phi(p,y,z):= \pi_{Y}(\mathcal{R}(\mathcal{A}(\bar \varphi + p \bar v + y)) - \mathcal{R}(\bar  \rho) -z),
\]
where $\mathcal{R}:(H^{1}(\Omega))' \to H^1(\Omega)$ is the Riesz isomorphism and $\pi_Y$ is the orthogonal
projection onto $Y$.
Since $\Phi(0,0,0)=0$ and
$
    \partial_y\Phi(0,0,0)
    =
    \pi_Y\mathcal R( D\mathcal{A}(\bar \varphi)\big|_Y)
$
is an isomorphism on $Y$, the implicit function theorem yields a small positive constant $\delta_0=\delta_0(\bar \rho, \bar \varphi)$ and an analytic map
$
    \bar y:
    [-\delta_0,\delta_0]\times
    \overline{U_{\delta_0}^Y(0)}
    \longrightarrow
    \overline{U_{\delta_0}^Y(0)}
$
such that
\[
    \Phi(p,\bar y(p,z),z)=0.
\]
We set
\begin{equation}\label{eq:varphipz}
    \varphi(p,z)
    :=
    \bar\varphi+p\bar v+\bar y(p,z).
\end{equation}
In particular, $
    \bar y(0,0)=0$ and $ \varphi(0,0)=\bar\varphi.$
Defining
\[
\theta(p,z):=\langle\mathcal{R}(\mathcal{A}( \varphi(p,z))) - \mathcal{R}(\bar  \rho) -z,  \bar v \rangle_{H^1(\Omega)},
\]
we have
\[
\mathcal{R}(\mathcal{A}( \varphi(p,z))) = \mathcal{R}(\bar  \rho)  +z +\theta(p,z) \bar v .
\]

For $(\alpha, z) \in [-\delta_0,\delta_0]\times \overline{U_{\delta_0}^Y(0)},$ let $\rho \in L^2(\Omega)$ be such that $\mathcal{R}(\rho-\bar \rho)= \alpha \bar v + z$. We define
\begin{align}\label{eq:defJ}
    \mathcal{J}_{\alpha,z}(p):= \mathcal{F}_{\rho}(\varphi(p,z)).
\end{align}

Observe that
\begin{equation}
\label{eq:reduced-derivative}
\mathcal{J}'_{\alpha,z}(p)= \theta(p,z)-\alpha,
\end{equation}
in particular
\begin{equation}\label{eq:reduced-derivative2}
    \mathcal J'_{\alpha,z}(p)=0
    \quad\Longleftrightarrow\quad
    \mathcal A(\varphi(p,z))=\rho.
\end{equation}
One can show that there exists $\delta_1, c>0$ and $\delta_2 \in (0,\delta_1/4)$ such that, up to taking $\delta_0$ smaller, the real analytic function $\mathcal{J}'_{\alpha,z}(p)$ is not identically zero in $[-\delta_1,\delta_1]$ and satisfies
\begin{align}
    &\mathcal{J}'_{\alpha,z}(p) \le -c \quad \text{ for } p \in [-\delta_1,-\delta_1+2\delta_2], \label{eq:left-sign}\\
    &\mathcal{J}'_{\alpha,z}(p) \ge c \quad \text{ for } p \in [\delta_1-2\delta_2,\delta_1].\label{eq:right-sign}
\end{align}

Finally, after possibly taking $\delta_0$ smaller, the
coercivity of $B_{\bar\varphi} |_{Y}$ yields the following property: if $w \in Y$ and $(p,z) \in [-\delta_0, \delta_0] \times \overline{U_{\delta_0}^Y(0)}$, then
\begin{equation}\label{eq:transverse-minimality}
    \mathcal F_\rho(\varphi(p,z))
    \le
    \mathcal F_\rho(\varphi(p,z)+w).
\end{equation}
Indeed, the first variation of $\mathcal F_\rho$ at
$\varphi(p,z)$ vanishes in every direction of $Y$, while the second
variation remains positive definite on $Y$.

Notice that the construction above applies both when the operator $B_{\bar \varphi}$ is degenerate and when it is nondegenerate. If the smallest eigenvalue of $B_{\bar \varphi}$ is strictly
positive, the equilibrium is locally unique by the implicit
function theorem. The following selection mechanism becomes relevant only
when the smallest eigenvalue vanishes.

\smallskip
\paragraph{\textit{ii) Selection mechanism.}}
Let $\rho_0\in L^m(\Omega)$ and assume that $\varphi_0$ is a global
minimizer of $\mathcal F_{\rho_0}$. We apply the construction in
$(i)$ with $\bar\rho=\rho_0$ and $\bar\varphi=\varphi_0$.
In particular, we have $ \varphi_0=\varphi(p_0,z_0)$ with $p_0=0$ and $z_0=0$.

We first set $\varphi_1^h = \varphi_0$.
Then $\varphi_1^h$ satisfies~\eqref{eq:nullvar} with $\rho=\rho_0$
and, trivially,
\begin{equation}\label{eq:first-phase-descent}
    \mathcal F_{\rho_0}(\varphi_1^h)
    \le
    \mathcal F_{\rho_0}(\varphi_0).
\end{equation}
We also have $\varphi_{1}^h = \varphi(p_{1}, z_{1})$
with $p_1=0$ and $z_1=0$.
Having fixed $\varphi_1^h$, we let $\rho_1^h$ be a minimizer of
\eqref{minpb} with $\varphi=\varphi_1^h$.

We next construct $\varphi_2^h$. By interpolation,~\eqref{eq:discomp} and the
uniform $L^2$ bound on the densities yield $\|\rho_1^h-\rho_0\|_{(H^1(\Omega))'}
    \to 0$ as $h\to 0.$
Therefore, for $0<h<h^*$ sufficiently small depending only on the initial data, there exist uniquely
$(\alpha_2, z_2) \in[-\delta_0,\delta_0]\times \overline{U_{\delta_0}^Y(0)}$ such that
\begin{equation}\label{eq:alpha2-z2}
    \mathcal R(\rho_1^h-\rho_0)
    =
    \alpha_2\bar v+z_2.
\end{equation}
Recall~\eqref{eq:defJ}, which in this case reads as
\[
    \mathcal J_{\alpha_2,z_2}(p)
    =
    \mathcal F_{\rho_1^h}(\varphi(p,z_2)).
\]
Starting from $p_1=0$, we define
\begin{equation}\label{eq:p2-selection}
    p_2:=
    \begin{cases}
        \displaystyle
        \inf\left\{
        p\in[p_1,\delta_1]:
        \mathcal J'_{\alpha_2,z_2}(p)>0
        \right\},
        &
        \mathcal J'_{\alpha_2,z_2}(p_1)<0,
        \\[1.2ex]
        p_1,
        &
        \mathcal J'_{\alpha_2,z_2}(p_1)=0,
        \\[1.2ex]
        \displaystyle
        \sup\left\{
        p\in[-\delta_1,p_1]:
        \mathcal J'_{\alpha_2,z_2}(p)<0
        \right\},
        &
        \mathcal J'_{\alpha_2,z_2}(p_1)>0.
    \end{cases}
\end{equation}
The sign conditions
\eqref{eq:left-sign}--\eqref{eq:right-sign} ensure that the sets in
\eqref{eq:p2-selection} are nonempty. By continuity,
\[
    \mathcal J'_{\alpha_2,z_2}(p_2)=0.
\]
We then set
\begin{equation}\label{eq:phi2-selection}
    \varphi_2^h:=\varphi(p_2,z_2).
\end{equation}
By~\eqref{eq:reduced-derivative2}, $\varphi_2^h$ satisfies
\eqref{eq:nullvar} with $\rho=\rho_1^h$.
Moreover, from
\eqref{eq:transverse-minimality} it follows that
\begin{equation}\label{eq:transverse-descent-step2}
    \mathcal F_{\rho_1^h}(\varphi(p_1,z_2))
    \le
    \mathcal F_{\rho_1^h}(\varphi(p_1,z_1))
    =
    \mathcal F_{\rho_1^h}(\varphi_1^h).
\end{equation}
On the other hand, by definition~\eqref{eq:p2-selection}, 
$p\mapsto\mathcal J_{\alpha_2,z_2}(p)$ is non-increasing along the
interval joining $p_1$ and $p_2$, in particular
\[
    \mathcal F_{\rho_1^h}(\varphi_2^h)
    =
    \mathcal J_{\alpha_2,z_2}(p_2)
    \le
    \mathcal J_{\alpha_2,z_2}(p_1)
    =
    \mathcal F_{\rho_1^h}(\varphi(p_1,z_2)).
\]
Hence, this together with~\eqref{eq:transverse-descent-step2} give
\begin{equation}\label{eq:second-phase-descent}
    \mathcal F_{\rho_1^h}(\varphi_2^h)
    \le
    \mathcal F_{\rho_1^h}(\varphi_1^h).
\end{equation}
Having selected $\varphi_2^h$, we finally let $\rho_2^h$ be a
minimizer of~\eqref{minpb} with $\varphi=\varphi_2^h$.

For the next steps, we proceed analogously as above. Assume that
\begin{align}
    \varphi_{n-1}^h = \varphi(p_{n-1},z_{n-1})
\end{align}
has already been selected and that $\rho_{n-1}^h$ is a minimizer of
\eqref{minpb} with $\varphi=\varphi_{n-1}^h$.

We now distinguish two cases. 

\emph{Continuation case.} As long as $\rho_{n-1}^h$ admits the following unique decomposition
\begin{equation}\label{eq:decomposition-n}
    \mathcal R(\rho_{n-1}^h-\rho_0)
    =
    \alpha_n\bar v+z_n \quad \text{ for some
$(\alpha_n,z_n) \in [-\delta_0,\delta_0]\times \overline{U_{\delta_0}^Y(0)}$,}
\end{equation}
we define
\begin{equation}\label{eq:pn-selection}
    p_n:=
    \begin{cases}
        \displaystyle
        \inf\left\{
        p\in[p_{n-1},\delta_1]:
        \mathcal J'_{\alpha_n,z_n}(p)>0
        \right\},
        &
        \mathcal J'_{\alpha_n,z_n}(p_{n-1})<0,
        \\[1.2ex]
        p_{n-1},
        &
        \mathcal J'_{\alpha_n,z_n}(p_{n-1})=0,
        \\[1.2ex]
        \displaystyle
        \sup\left\{
        p\in[-\delta_1,p_{n-1}]:
        \mathcal J'_{\alpha_n,z_n}(p)<0
        \right\},
        &
        \mathcal J'_{\alpha_n,z_n}(p_{n-1})>0,
    \end{cases}
\end{equation}
and set
\begin{equation}\label{eq:phin-selection}
    \varphi_n^h:=\varphi(p_n,z_n).
\end{equation}
Once again, by~\eqref{eq:reduced-derivative2} we have~\eqref{eq:nullvar}, i.e.,
\begin{equation} \label{eq:nullvar2}
    \mathcal A(\varphi_n^h)=\rho_{n-1}^h.
\end{equation}
Moreover,~\eqref{eq:transverse-minimality} and~\eqref{eq:pn-selection} yield
\begin{equation}\label{eq:phase-descent-n}
    \mathcal F_{\rho_{n-1}^h}(\varphi_n^h)
    \le
    \mathcal F_{\rho_{n-1}^h}(\varphi_{n-1}^h).
\end{equation}
Finally, having selected $\varphi_n^h$, we let $\rho_n^h$ be a
minimizer of~\eqref{minpb} with $\varphi=\varphi_n^h$.

\emph{Restart case.}
Suppose that~\eqref{eq:decomposition-n} does not hold for $\rho_{n-1}^h$. In this case, we
restart the construction by choosing
\begin{equation}\label{eq:restart-selection}
    \varphi_n^h
    \in
    \operatorname*{argmin}_{\varphi\in H^1(\Omega)}
    \mathcal F_{\rho_{n-1}^h}(\varphi).
\end{equation}
In particular, $\varphi_n^h$ satisfies~\eqref{eq:nullvar} with
$\rho=\rho_{n-1}^h$. Moreover, since $\varphi_{n-1}^h$ is an
admissible competitor in~\eqref{eq:restart-selection},
\begin{equation}\label{eq:restart-phase-descent}
    \mathcal F_{\rho_{n-1}^h}(\varphi_n^h)
    \le
    \mathcal F_{\rho_{n-1}^h}(\varphi_{n-1}^h).
\end{equation}
We then let $\rho_n^h$ be a minimizer of~\eqref{minpb} with
$\varphi=\varphi_n^h$.

We then restart the procedure from $(\rho_{n-1}^h,\varphi_n^h)$, which plays the role of the initial pair $(\rho_0,\varphi_0)$ in the construction above.
In particular, we have
\[
    \varphi_n^h=\varphi(0,0),
\]
so that 
$
    p_n=0
$ and $
    z_n=0.
$

\paragraph{\textit{iii) Finite number of restarts.}}
We next show that the local construction above can be iterated on the whole time interval with a finite number of restarts which is uniform with respect to $h$. Let
\[
    \mathcal K_M
    :=
    \left\{
        \rho\in\mathcal  K: 
        \|\rho\|_{L^2(\Omega)}\le M
    \right\},
\]
where $M>0$, and define the set of possible restart points by
\begin{equation}\label{eq:restart-set}
    \mathcal{C}_M
    :=
    \Big\{
        (\rho,\varphi)\in
        \mathcal K_M\times H^1(\Omega):
        \varphi\in
        \operatorname*{argmin}_{\psi\in H^1(\Omega)}
        \mathcal F_\rho(\psi)
    \Big\}.
\end{equation}

\begin{lemma}
\label{lem:uniform-restarts}
Let $M>0$. Then the set $\mathcal C_M$ is compact in
$(H^1(\Omega))'\times H^1(\Omega)$. Moreover, there exist
$\tau_*>0$ and $h_*>0$ depending only on the initial data such that, for every $h<h_*$, whenever a
restart occurs at step $n_*$, no further restart occurs at any step
$n\ge n_*$ satisfying
\[
    (n-n_*)h\le\tau_*.
\]
Consequently, the number of restarts on $[0,T]$ is bounded
independently of $h$.
\end{lemma}

\begin{proof}
We divide the proof into three steps.

\noindent \emph{Step 1.}
We first prove that $\mathcal{C}_M$ is compact in $(H^{1}(\Omega))'\times H^1(\Omega).$

Let $(\rho_j,\varphi_j)_{j \in \mathbb{N}}\subset\mathcal{C}_M$.
Since $(\rho_j)_{j \in \mathbb{N}}$ is uniformly bounded in $L^2(\Omega)$, up to a subsequence,
there exists $\rho\in \mathcal K_M$ such that
\begin{equation}\label{eq:restart-rho-compact}
    \rho_j\to\rho
    \qquad\text{strongly in }(H^{1}(\Omega))'.
\end{equation}
Since $\varphi_j$ is a global minimizer
of $\mathcal F_{\rho_j}$, we have
\[
    \mathcal A(\varphi_j)=\rho_j,
\]
which combined with 
\eqref{eq:restart-rho-compact} and the properness of $\mathcal A$
yields, up to a subsequence,
\begin{equation}\label{eq:restart-phi-compact}
    \varphi_j\to\varphi
    \qquad\text{strongly in }H^1(\Omega)
\end{equation}
and $
    \mathcal A(\varphi)=\rho. $
Moreover, since 
\[
    \mathcal F_{\rho_j}(\varphi_j)
    \le
    \mathcal F_{\rho_j}(\tilde \varphi), \qquad \text{for every } \tilde \varphi\in H^1(\Omega),
\]
 using
\eqref{eq:restart-rho-compact} and
\eqref{eq:restart-phi-compact}, we obtain
that $ \varphi\in\operatorname*{argmin}\mathcal F_\rho,$
hence
$
    (\rho,\varphi)\in\mathcal{C}_M.
$

\noindent
\textit{Step 2.}
For every
\[
    (\bar\rho,\bar\varphi)\in\mathcal{C}_M,
\]
the local analysis in {\textit{(i)}} provides a constant
$\delta_0=\delta_0(\bar\rho,\bar\varphi)>0$ such that, whenever
\[
    \|\rho-\bar\rho\|_{(H^{1}(\Omega))'}<\delta_0,
\]
the selection procedure can be carried out and yield a function
$\varphi=\varphi(\rho)$  and a $\delta_0$-neighbourhood $\mathcal{U}_{\delta_0}(\bar \rho, \bar \varphi)$ of $(\bar \rho, \bar \varphi)$.

Consider the open cover of $\mathcal{C}_M$ given by $\left\{
        \mathcal{U}_{\delta_0/2}(\bar \rho, \bar \varphi)
        :(\bar\rho,\bar\varphi)\in\mathcal{C}_M
    \right\}$. 
By compactness, we may extract a finite
subcover and, denoting by
$\delta_0^1,\ldots,\delta_0^N$ the corresponding radii, we set $\delta_*:=\min_{1\le j\le N} \delta_0^j/2$.

\noindent
\textit{Step 3.}
It remains to show that $\rho^h_n$ cannot move by
$\delta_*$ in $(H^{1}(\Omega))'$ in an arbitrarily short time.

We first recall~\eqref{eq:discomp}, which can be written as, for $m>n$,
\begin{align*}
    W_2(\rho_n^h,\rho_m^h) \le
    C\sqrt{(m-n)h}.
    \label{eq:atlas-W2-modulus}
\end{align*}
On the other hand, by interpolation, there exists a non-decreasing modulus of continuity
$\omega_M:[0,+\infty)\to[0,+\infty)$, with
$\omega_M(s) \to  0$ as $s\downarrow0$, such that
\begin{equation*}
    \|\rho-\tilde\rho\|_{(H^{1}(\Omega))'}
    \le
    \omega_M\bigl(W_2(\rho,\tilde\rho)\bigr)
    \qquad
    \text{for every }
    \rho,\tilde\rho\in\mathcal K_M.
\end{equation*}
Hence, we deduce
\begin{equation}\label{eq:atlas-Hminus1-modulus}
    \|\rho_n^h-\rho_m^h\|_{(H^{1}(\Omega))'}
    \le
    \omega_M (C\sqrt{(m-n)h}).
\end{equation}

We now choose $\tau_*>0$ such that
$
    \omega_M(C\sqrt{\tau_*})
    <
    \delta_*.
$
Let $h_*>0$ be the constant given by the selection procedure and assume that $h_*\le\tau_*$.
Suppose that a restart occurs at step $n_*$. Since
\[
    (\rho_{n_*-1}^h,\varphi_{n_*}^h)\in\mathcal C_M,
\]
there exists a $\mathcal{U}_{\delta_0}(\bar \rho, \bar \varphi)$ from \textit{Step 2} such that $ (\rho_{n_*-1}^h,\varphi_{n_*}^h) \in \mathcal{U}_{\delta_0}(\bar \rho, \bar \varphi)$.
Fixed $h \le h^*$, for every $n\ge n_*$ such that
$
    (n-n_*)h\le\tau_*,
$
we have, by~\eqref{eq:atlas-Hminus1-modulus},
\[
    \|\rho_{n-1}^h-\rho_{n_*-1}^h\|_{(H^{1}(\Omega))'}
    \le
    \omega_M(C\sqrt{(n-n_*)h})
    <
    \delta_*.
\]
In particular, we observe that no
restart can occur for $n \ge n_*$ such that
$
    (n-n_*)h\le\tau_*.
$
Moreover, if
$
    n_1<n_2<\cdots<n_{N_h}
$
are the restart indices, then
\[
    (n_{k+1}-n_k)h>\tau_*
    \qquad
    \text{for }k=1,\ldots,N_h-1.
\]
Consequently, 
\[
    N_h
    \le
    1+\left[\frac{T}{\tau_*}\right],
\]
which is independent of $h$.
\end{proof}

\paragraph{\textit{iv) Strong convergence of the selected phases.}}

\begin{proposition}
\label{prop:sel-phase-compactness}
Given~$T>0$, $\rho_0 \in L^m(\Omega)$ such that $\int_\Omega \rho_0 \,\dx=1$ and $\varphi_0$ a minimizer of the functional $\varphi \mapsto \mathcal{E}_{\textnormal{int}, \eps}(\varphi) +\int_{\Omega} \varphi (\lambda-\rho_0) \, \dx $, let $(\rho_n^h,\varphi^h_n)_{n \in \mathbb{N}}$  be a discrete solution starting from $(\rho_0, \varphi_0)$ and constructed via the selection procedure above.
Then we have
\begin{equation}\label{eq:strongconvvarphi}
    \varphi^h\to\varphi
    \quad\text{strongly in }
    L^2(0,T;H^1(\Omega))
    \cap L^4(\Omega\times(0,T)).
\end{equation}
In particular,
\begin{equation}\label{eq:convener}
    \int_0^T\!\int_\Omega
    \left(
        \frac{\varepsilon}{2}|\nabla\varphi^h|^2
        +\frac{1}{\varepsilon}F(\varphi^h)
    \right)\dx\dt
    \to
    \int_0^T\!\int_\Omega
    \left(
        \frac{\varepsilon}{2}|\nabla\varphi|^2
        +\frac{1}{\varepsilon}F(\varphi)
    \right)\dx\dt.
\end{equation}
\end{proposition}

\begin{proof}
By Lemma~\ref{lem:uniform-restarts}, we have that the number $N_h$ of restarts  on
$[0,T]$ is uniformly bounded with respect to $h$. Hence, up to a
subsequence, the corresponding restart times $t_1^h<\ldots<t_{N_h}^h$ converge as $h \to 0$ to finitely many
points $t_1^\ast<\ldots<t^\ast_{N}$ in $[0,T]$. It is therefore enough to prove the assertion on
each open interval between two consecutive limiting restart times.

Fix such an interval and $t_0$ in its interior. For $h>0$ sufficiently
small, the selected phases $\varphi
^h$ are given by the local construction introduced in \textit{(i)}.
We then consider the functional   $p\mapsto \mathcal J_{\alpha_0,z_0}(p) $ for some $(\alpha_0,z_0)$ associated with  $\rho(t_0)$, and set
\[
    g_0(p):=\mathcal J'_{\alpha_0,z_0}(p).
\]
By the properties established in {\textit{(i)}}, $g_0$ is real
analytic and not identically zero in $[-\delta_1, \delta_1]$. Hence it has only finitely many
zeros $-\delta_1<q_1<\ldots<q_k<\delta_1$ for some $ k\in \mathbb{N}$.
We choose $\eta>0$ sufficiently small such
that the intervals
\[
    Q_i:=(q_i-\eta,q_i+\eta),
    \qquad i=1,\ldots,k,
\]
are pairwise disjoint and compactly contained in
$(-\delta_1,\delta_1)$. Since $g_0$ does not vanish outside these
intervals, there exists $c_0>0$ such that
\begin{equation}\label{eq:sel-away-roots}
    |g_0(p)|\ge c_0
    \qquad\text{for every }
    p\in[-\delta_1,\delta_1]\setminus
    \bigcup_{i=1}^k Q_i .
\end{equation}

By the strong convergence in~\eqref{eq:convrhoH-1}, we have
\[
    (\alpha^h,z^h)\to(\alpha,z)
    \qquad\text{uniformly on }[0,T],
\]
where $(\alpha^h(t),z^h(t))$ and $(\alpha(t),z(t))$ are the coordinates associated with $\rho^h(t)$ and $\rho(t)$, respectively.
Moreover, since $\rho$ is continuous in $(H^{1}(\Omega))'$,
\[
    (\alpha(t),z(t))\to(\alpha_0,z_0)
    \qquad\text{as }t\to t_0.
\]
Therefore, by the continuity of $\mathcal J'_{\alpha, z}(p)$ with respect to $(\alpha, z, p)$, we can choose a sufficiently small neighbourhood
$U$ of $t_0$ and then $h>0$ sufficiently small such that
\[
    \sup_{t\in U}\sup_{p\in[-\delta_1,\delta_1]}
    \left|
        \mathcal J'_{\alpha^h(t),z^h(t)}(p)-g_0(p)
    \right|
    <\frac{c_0}{2}.
\]
Since every $p_n^h$ given by~\eqref{eq:pn-selection} satisfies $\mathcal J'_{\alpha_n^h,z_n^h}(p_n^h)=0,$
it follows from~\eqref{eq:sel-away-roots} that, for $h$ sufficiently
small and every $n$ such that $nh\in U$, \[p_n^h\in\bigcup_{i=1}^k Q_i.\]
Since the intervals $Q_i$ are pairwise disjoint, there exists a unique
$i_n^h\in\{1,\ldots,k\}$ such that $p_n^h\in Q_{i_n^h}.$

By the same argument as in~\cite[Lemma~3.2]{Schaetzle},
the sequence $n\mapsto i_n^h$ is monotone as long as $nh\in U$.
On the other hand, given \[
    \rho_-^h(t):=\rho_{n-1}^h
    \qquad\text{for }t\in((n-1)h,nh],
\]
equation~\eqref{eq:nullvar2} reads as 
\[
    \mathcal A(\varphi^h(t))=\rho_-^h(t).
\]
Hence, the convergence of $\rho_-^h(t)$ in $(H^{1}(\Omega))'$, and the
properness of $\mathcal A$ imply that $(\varphi^h(t))_{h>0}$
is relatively compact in $H^1(\Omega)$ for every fixed $t \in U$.
By a diagonal argument,
$\varphi^h(t)$ converges strongly in $H^1(\Omega)$ for every 
$t \in (\mathbb{Q} \cap (0,T))\setminus\{t_1^*,\ldots,t_N^*\}$.
Finally, using~\cite[Lemma~3.3]{Schaetzle}, we can conclude that
\[
    \varphi^h(t)\to\varphi(t)
    \qquad\text{strongly in }H^1(\Omega) \text{  for a.e. } t\in (0,T).
\]
By~\eqref{eq:coercivity}, the dominated
convergence theorem then yields~\eqref{eq:strongconvvarphi}. 
\end{proof}

\paragraph{\textit{v) Phase almost minimizing property.}}
Observe that
\[
    \rho \mapsto
    \inf_{\tilde \varphi\in H^1(\Omega)}\mathcal F_\rho(\tilde \varphi)
\]
is continuous on $\mathcal K_M$
with respect to the $(H^{1}(\Omega))'$-topology. 
Fix a restart point
$(\bar\rho,\bar\varphi)\in\mathcal C_M$. By
the continuity of $\mathcal F_\rho(\varphi)$ with respect to
$(\rho,\varphi)$, for every $\kappa>0$ there exists a sufficiently small $\delta_0$ given by \textit{(i)} such that,
if $(\rho, \varphi) $ belongs to a $\delta_0$-neighborhood of $(\bar\rho,\bar\varphi)$, then
\[
    \mathcal F_\rho(\varphi)
    -
    \inf_{\tilde \varphi\in H^1(\Omega)}\mathcal F_\rho(\tilde \varphi)
    <\kappa.
\]
By the compactness of $\mathcal C_M$, 
we can conclude that, for every $n$,
\begin{equation}\label{eq:sel-kappa-discrete}
    \mathcal F_{\rho_{n-1}^h}(\varphi_n^h)
    \le
    \inf_{\tilde \varphi\in H^1(\Omega)}
    \mathcal F_{\rho_{n-1}^h}(\tilde \varphi)
    +\kappa .
\end{equation}
In particular, for every $\tilde\varphi\in L^1(0,T;H^1(\Omega))$,
\[
    \mathcal F_{\rho_-^h(t)}(\varphi^h(t))
    \le
    \mathcal F_{\rho_-^h(t)}(\tilde\varphi(t))+\kappa
    \qquad\text{for a.e. }t\in(0,T).
\]
Integrating in time and passing to the limit as $h\downarrow0$ by
the strong convergences~\eqref{eq:strongcomprhop} and~\eqref{eq:strongconvvarphi}, we obtain
\[
    \int_0^T\mathcal F_{\rho(t)}(\varphi(t))\, \dt
    \le
    \int_0^T\mathcal F_{\rho(t)}(\tilde\varphi(t))\, \dt
    +T\kappa .
\]
In particular, one may choose $\kappa = \eps/T$, thus obtaining~\eqref{eq:almostmin}.

\subsection{Convergence to the continuous phase-field problem} \label{sec:continuouslimit}

\subsubsection{Construction of the velocity field and transport equation}\label{sc:discretevel}  In order to obtain~\eqref{eq:transport} we argue as in~\cite{KLM}, in particular we refer the reader to~\cite[Section 6.1]{KLM} for more details.
 
Defining $\mathbf{u}^h := \nabla \phi^h$ and $\mathbf{v}^h := \sqrt{\rho^h} \nabla \phi^h $, from~\eqref{phibound} we deduce that (up to a subsequence)
\begin{equation} \label{eq:limitvh}
\sqrt{\rho^h} \mathbf{u}^h = \mathbf{v}^h \rightharpoonup \mathbf{v} \quad \text{weakly in } L^2(\Omega \times (0,T)).
\end{equation}
Moreover, the measure $\sqrt{\rho}\mathbf{v}  \mathcal{L}^d \llcorner \Omega \otimes \mathcal{L}^1 \llcorner (0,T)$  is absolutely continuous with respect to $ \rho  \mathcal{L}^d \llcorner \Omega \otimes \mathcal{L}^1 \llcorner (0,T)$, in particular there exists a vector field $\mathbf{u} \in L^1( \Omega \times (0,T); \rho \,\dx\dt )$ such that $\sqrt{\rho}\mathbf{v}  =  \rho\mathbf{u}  $ in $L^1(\Omega \times (0,T))$ and
\begin{equation}\label{eq:identificationv}
    \mathbf{v}  = \sqrt{\rho}\mathbf{u}  \quad \text{on }\supp \rho.
\end{equation}

Recalling that $(\operatorname{Id}-h \nabla \phi^h_n)_\# \rho^h_n = \rho^h_{n-1}$, for $\eta \in C_c^\infty(\overline\Omega \times [0, +\infty))$ it holds
\begin{align*}
\frac1h \int_\Omega (\rho^h_n - \rho^h_{n-1}) \eta \, \dx = \frac1h \int_\Omega \rho^h_n(x) (\eta(x)- \eta(x-h \nabla \phi^h_n(x)) \, \dx .
\end{align*}
By Taylor's expansion on the right-hand side, we have
\begin{align}\label{eq:equationrho}
\frac1h \int_\Omega (\rho^h_n - \rho^h_{n-1}) \eta \, \dx =  \int_\Omega \rho^h_n(x) \nabla \eta \cdot   \nabla \phi^h_n  \, \dx  + \mathcal{R},
\end{align}
where the error $\mathcal{R}$, once integrated in time, vanishes as $h \rightarrow 0$.
Therefore, integrating in time~\eqref{eq:equationrho}, using~\eqref{eq:strongcomprhop},~\eqref{eq:limitvh} and~\eqref{eq:identificationv}, we obtain~\eqref{eq:transport} in the limit $h \rightarrow 0$.

\subsubsection{Passage to the limit in the discrete equation} 
From the convergences~\eqref{eq:strongcomprhop} and~\eqref{eq:limitvh} it follows that
\begin{equation*}
    \rho^h \nabla \phi^h \to \sqrt{\rho}\textbf{v} \quad \text{ weakly in } L^{1}(\Omega \times (0,T)) ,
\end{equation*}
as $ h \to 0$.
Hence, the left-hand side of~\eqref{eq:eulerlagrdis} passes to the desired limit as $h \rightarrow 0$.

Moreover, we can pass to the limit in the terms on the right-hand side of~\eqref{eq:eulerlagrdis} thanks to Lemma~\ref{lemma:compactness rho},~\eqref{eq:strongcomprhop}
~\eqref{eq:strongconvvarphi}, and~\eqref{eq:convener}. In particular, we have that $\mu \geq \rho^m \mathcal{L}^d \llcorner \Omega \otimes \mathcal{L}^1 \llcorner (0,T)$. Indeed, consider a nonnegative test function $\eta \in C^\infty_c(\Omega\times (0,T))$, using the almost everywhere convergence in $\Omega \times (0, T)$ given by~\eqref{eq:strongcomprhop} and applying Fatou's lemma, we obtain \[
\int_{\Omega}\int_0^T  \rho^m \eta\, \dx \dt  \leq \liminf_{h \to 0} \int_{\Omega}\int_0^T  (\rho^h )^m \eta \, \dx \dt = \int_{\Omega}\int_0^T \eta \, \mathrm{d}\mu.
\]

\subsubsection{Passage to the limit in the Allen-Cahn equation}
Each term of equation~\eqref{eq:ACopt} passes to the desired limit due to~\eqref{eq:weakcomprho},~\eqref{eq:strongconvvarphi}, hence yielding~\eqref{eq:ACweak}.

\subsubsection{Energy dissipation rate}
Similarly as in~\cite{Chambolle-Laux} (see also~\cite{AGS}), it  can be shown
that the De Giorgi's variational interpolation $\tilde \rho^h$, defined as 
\begin{align}\label{eq:degiorgimin}
\tilde\rho^h(t)  \in  \arg \min_{\rho \in \mathcal{K}} \Big\{ \frac{W^2_2(\rho, \rho^h_{n-1})}{2(t-(n-1)h)}  + \mathcal{E}_\varepsilon(\rho, \varphi_n^h) \Big\}
\text{ if } t  \in ((n-1)h,nh],
\end{align}
satisfies 
\begin{equation} \label{dgine}
\begin{split}
    &\frac{h}{2} \left (\frac{W_2(\rho_n^h, \rho_{n-1}^h)}{h} \right)^2 + \frac{1}{2} \int_{0}^{h} \left (\frac{W_2(\tilde \rho^h((n-1)h+t),\rho_{n-1}^h)}{t} \right )^2 \text{d}t \\
    &\le \mathcal{E}_\varepsilon(\rho_{n-1}^h,\varphi_{n}^h)-\mathcal{E}_\varepsilon(\rho_{n}^h,\varphi_n^h) \le \mathcal{E}_\varepsilon(\rho_{n-1}^h,\varphi_{n-1}^h)-\mathcal{E}_\varepsilon(\rho_{n}^h,\varphi_n^h),
\end{split}
\end{equation}
where the last inequality follows from~\eqref{eq:phase-descent-n} and~\eqref{eq:restart-phase-descent}.
Notice that $\tilde \rho^h (n h ) = \rho^h_{n} = \rho^h(nh)$, for all $n \in \mathbb{N}_0$.

Summing in~\eqref{dgine} over $n$ from $n_0 +1$ to $n_1$, we obtain
\begin{equation}
\begin{split}
    &\frac{h}{2} \sum_{n=n_0+1}^{n_1} \left (\frac{W_2(\rho_n^h, \rho_{n-1}^h)}{h} \right)^2 + \frac{1}{2} \sum_{n=n_0+1}^{n_1}  \int_{0}^{h} \left (\frac{W_2(\tilde \rho^h((n-1)h+t),\rho_{n-1}^h)}{t} \right )^2 \text{d}t \\
    &\le \mathcal{E}(\rho_{n_0}^h,\varphi_{n_0}^h)
    - \mathcal{E}(\rho_{n_1}^h,\varphi_{n_1}^h) \le \mathcal{E}(\rho_{0},\varphi_{0})
    - \mathcal{E}(\rho_{n_1}^h,\varphi_{n_1}^h).
    \label{eq:discdiss}
\end{split}
\end{equation}
First, recalling~\eqref{eq:kanto}-\eqref{phibound} and, in particular, the convergence~\eqref{eq:limitvh}, we can pass to the limit as $h \to 0$ in the first term on the left-hand side of~\eqref{eq:discdiss}.
Then, note that, by a change of variable, it holds
\[
 \frac{1}{2}\sum_{n=n_0+1}^{n_1} \int_{0}^{h} \left (\frac{W_2(\tilde \rho^h((n-1)h+t),\rho_{n-1}^h)}{t} \right )^2  \text{d}t =  \frac{1}{2}\int_{n_0 h}^{n_1 h} \left (\frac{W_2(\tilde \rho^h(t),\rho^h(t))}{t-[t/h]h} \right )^2   \text{d}t ,
\]
and
\[
\frac{W^2_2(\tilde \rho^h(t),\rho^h(t))}{2(t-[t/h]h)}  = \frac{t-[t/h]h}{2} \int_{\Omega} |\nabla \tilde \phi^h(t)|^2 \tilde \rho^h \, \dx,
\]
where $\tilde \phi^h$ is the optimal Kantorovich potential in $W_2(\tilde \rho^h(t),\rho^h(t))$.
For $n_0 = 0$, $n_1=N$, and $\tau=Nh \in (0,T)$, it follows that 
\[
\frac{1}{2}\sum_{n=1}^{N} \int_{0}^{h} \left (\frac{W_2(\tilde \rho^h((n-1)h+t),\rho_{n-1}^h)}{t} \right )^2 \text{d}t = \frac{1}{2} \int_{0}^{\tau} \int_{\Omega} |\nabla \tilde \phi^h(t)|^2 \tilde \rho^h \,  \dt\dx.
\]
We deduce that also $\mathbf{\tilde v}^h := \sqrt{\tilde \rho^h} \nabla \tilde \phi^h $ weakly converges in $(L^2(\Omega \times (0,T)))^d$ to the function $\mathbf{v}$ given by~\eqref{eq:limitvh}.
Indeed, arguing similarly to the convergence results established for $\rho^h$, one can show that for almost every $t \in (0,T)$ it holds
$\tilde \rho^h(t) \to \tilde \rho(t)$ in $L^1(\Omega)$, therefore one can conclude using the analogue of~\eqref{eq:eulerlagrdis} for $(\tilde \rho^h(t), \varphi^h(t))$ together with~\eqref{eq:strongconvvarphi}.
Moreover, since $\tilde \rho^h(t) \to \rho(t)$ strongly in the $2$-Wasserstein distance due to~\eqref{eq:degiorgimin} and the equivalent of~\eqref{eq:energydecr} (see also equation 3.2.7~\cite{AGS}), we can identify $\tilde \rho = \rho$. 
Hence, we can also pass to the limit in the second term on the left-hand side of~\eqref{eq:discdiss} and observe that
\begin{align*}
\int_0^{\tau} \int_\Omega |\mathbf{v}|^2\, \dx \dt 
\geq \int_{\supp \rho} |\mathbf{v}|^2\, \dx \dt = \int_{\supp \rho} \rho |\mathbf{u}|^2\, \dx \dt 
= \int_0^{\tau} \int_\Omega \rho |\mathbf{u}|^2\, \dx \dt,
\end{align*}
where we used~\eqref{eq:identificationv}.

To conclude about the validity of~\eqref{eq:energydiss} we use the lower semicontinuity of the energies $\mathcal{E}(\rho^h(\cdot,t),\varphi^h(t))$ for almost every $t \in (0,T)$ due to~\eqref{eq:coercivity},~\eqref{eq:strongcomprhop}, and~\eqref{eq:strongconvvarphi}.


\section{The sharp interface limit ($\eps \to 0$)}\label{sec:sharplimit}
In this section, we prove that weak solutions of the phase-field model in the sense of Definition~\ref{def:weaksol} converge, as the parameter $\eps \to 0$, to weak solutions of the Verigin problem with phase transition as introduced in~\cite{KLM}. We establish this result by showing strong compactness properties for $(\rho_\eps,\varphi_\eps)_{\eps>0}$ and then passing to the limit in the distributional equations~\eqref{eq:transport}-\eqref{eq:ACweak}.

The starting point of~\cite{KLM} is a minimizing movement scheme for the Verigin problem with phase transition. In particular, for $\rho \in L^m(\Omega)$ and $E \subset \Omega$ a set of finite perimeter, the evolution is driven by the energy
\begin{align}\label{eq:sharpenergy}
\mathcal{E}(\rho, \chi_E) := \sigma \int_{\Omega} 1 \, \mathrm{d}|\nabla \chi_E |  +   \int_{\Omega}    \frac{1}{m-1}\rho^m \,\dx+  \int_{\Omega}\chi_E ( \lambda   - \rho )  \, \mathrm{d}x,
\end{align}
where $\sigma>0$ denotes the surface tension. 
Similarly as done in Section~\ref{sec:existence}, the authors construct discrete solutions via a minimizing movement scheme for the energy~\eqref{eq:sharpenergy}, using the $2$-Wasserstein distance as dissipation.
They then show that, as the time discretization step $h \to 0$, these solutions converge to weak solutions of the Verigin problem with phase transition in the sense of the following definition.

\begin{definition}[Weak solution to the Verigin problem with phase transition]
     \label{def:sharpsol}
Let $T>0$ and $\Omega \subset \mathbb{R}^d$ be a smooth open set.    Let $E_0 \subset \Omega$ be a measurable set and $\rho_0 \in L^1(\Omega)$ be initial conditions such that $\int_\Omega \rho_0=1$ and $\mathcal{E}(\rho_0,\chi_{E_0})< \infty$. A triple of functions $(\rho, \chi, \mathbf{u})$ is a \textit{distributional solution to the Verigin problem with phase transition} if $\rho \in L^{m}(\Omega \times (0,T))$,  $\chi \in   L^\infty(0,T;BV(\Omega))$ with $0\le \chi \le 1$,  $\mathbf{u} \in (L^2( \Omega \times (0,T); \rho \, \dx\dt ))^{ d}$, and satisfies the following properties:
\begin{itemize}
    \item[i)] For all $\eta \in C^\infty_c(\overline\Omega \times [0, T))$, we have 
    \begin{equation}\label{eq:transportsharp}
        - \int_0^T \int_{\Omega} \rho (\partial_t \eta + \mathbf{u}  \cdot \nabla \eta ) \, \dx \dt = \int_{\Omega} \rho_0 \eta(\cdot, 0) \, \dx .
    \end{equation}
     \item[ii)]  The distributional Young-Laplace law
    \begin{align}\label{eq:weaksolsharp}
       & - \int_0^T \int_{\Omega} \rho \mathbf{u} \cdot \xi \,\mathrm{d} x \dt  \\
      \notag  &= \sigma \int_0^T \int_{\Omega} (\operatorname{div} \xi - \nu  \cdot \nabla \xi \nu  )
   \, \mathrm{d} |\nabla \chi| \dt 
      -\int_0^T \int_{\Omega} \mathrm{div} \xi\, \mathrm{d} \mu + \int_0^T \int_{\Omega} \lambda\chi \mathrm{div} \xi \,  \dx \dt
    \end{align}
    is satisfied for all $\xi \in C_c^\infty(\overline \Omega \times [0, T); \mathbb{R}^d) $ with $\xi \cdot \nu_{ \Omega} =0$ along $\partial \Omega \times (0,T)$, where $\nu $ denotes the Radon-Nikodym derivative $-\frac{ \mathrm{d} \nabla \chi}{\mathrm{d} |\nabla \chi|}$ and $\mu \in \mathcal{M}_b(\Omega \times (0,T))$ is such that $\mu \ge \rho^m \mathcal{L}^d \llcorner \Omega \otimes \mathcal{L}^1 \llcorner (0,T)$.
    \item[iii)] The energy dissipation rate 
    \begin{align}\label{eq:energydisssharp}
    \mathcal{E}(\rho(\tau), \chi(\tau)) + \int_0^\tau \int_\Omega \rho |\mathbf{u}|^2\, \dx \dt \leq \mathcal{E}(\rho_0, \chi_{E_0})
    \end{align}
    holds for almost every $\tau\in (0,T)$.
\end{itemize}
\end{definition}

We now turn to the main result of this section and prove that weak solutions to the phase-field model~\eqref{eq:phasefield} in the sense of Definition~\ref{def:weaksol} satisfying the almost-minimizing property~\eqref{eq:almostmin} converge to weak solutions of the Verigin problem with phase transition in the sense of Definition~\ref{def:sharpsol}. In particular, this result applies to the solutions constructed in Section~\ref{sec:existence} (cf.~Theorem~\ref{theo:existence}) and reads as follows.

\begin{theorem}\label{theo:sharplimit}
  Let $m\geq 2$ and $d \in \{2,3\}$. Let $(\rho_\eps, \varphi_\eps,\mathbf{u}_\eps)$ be weak solutions to the phase-field model~\eqref{eq:phasefield}  in the sense of Definition~\ref{def:weaksol} with $\mu_\eps = \rho_\eps^m \mathcal{L}^d \llcorner \Omega \otimes \mathcal{L}^1 \llcorner (0,T)$ satisfying the almost-minimizing property~\eqref{eq:almostmin}, and
  with well-prepared initial data $(\rho_{\eps,0}, \varphi_{\eps,0})$, i.e., satisfying 
  \begin{subequations}
  \begin{align}
  \rho_{\eps,0}&\rightarrow\rho_{0}\quad\text{in $L^m(\Omega)$,} \label{eq:convrhoeps0}\\
\varphi_{\eps,0}&\rightarrow\chi_{E_0}\quad\text{in $L^1(\Omega)$,} \label{eq:convphieps0}\\
    \mathcal{E}_{\textnormal{int}, \eps}(\varphi_{\eps,0}) = \int_{\Omega} \frac\varepsilon2 |\nabla \varphi_{\eps,0}|^2 + \frac1\varepsilon F(\varphi_{\eps,0}) \, \mathrm{d}x &\rightarrow\sigma |\nabla \chi_{E_0}|(\Omega), \label{eq:conveneps0}
\end{align}
\end{subequations}
where $\sigma := \int_0^1\sqrt{2F(s)}\,\mathrm{d}s$,
for some $\rho_0 \in L^m(\Omega)$ and $E_0 \subset \Omega$ of finite perimeter in $\Omega$.
 Then,  under the assumption $\rho_\eps \geq c $ for a fixed constant $c>0$,  there exists a subsequence $\eps_j\rightarrow 0$ such that the
  following properties hold.
  \begin{enumerate}
  \item[(i)] There exists $\rho \in L^{m}(\Omega \times (0,T))$ such that
   \begin{equation}\label{eq:convrhoeps}
\rho_{\eps_j}\rightarrow \rho \qquad\text{strongly in } L^p(\Omega \times (0,T)), \quad \text{ for every } 1\le p<m.
    \end{equation}
  \item[(ii)] 
  There exists $\chi \in L^{\infty}(0,T;BV(\Omega))$ 
  such that
  \begin{equation} \label{eq:convarphieps0}
\varphi_{\varepsilon_j} \to\chi 
\quad\text{ strongly in } L^1(0,T;L^1_0(\Omega)).
  \end{equation}
Moreover, there exist a function $\tilde\chi\in L^\infty(0,T;BV(\Omega;\{0,1\}))$ and a measurable set $R\subset[0,T]$ such that $\tilde\chi(t)=\chi(t)$ for $t\in R$, while $\tilde\chi(t)\equiv c\in\{0,1\}$ for $t\notin R$. In particular, $\tilde\chi(t)=\chi_{E(t)}$ for a family of sets $E(t)\subset\Omega$ of finite perimeter.
  \item[(iii)] There exists $\mathbf{u} \in (L^2( \Omega \times (0,T); \rho \, \dx\dt ))^{ d}$  such that
    \begin{equation}     \label{eq:convweigthedueps}
    \sqrt{\rho_{\eps_j}}\mathbf{u}_{\eps_j}\rightarrow \sqrt{\rho} \mathbf{u}
      \quad\text{weakly in } (L^2( \Omega \times (0,T))^{ d}.
    \end{equation}
      \item[iv)] The time-integrated interfacial phase-field energy converges in the sense that
      \begin{equation}\label{eq:energyconv}
          \int_0^T  \mathcal{E}_{\textnormal{int}, \eps}(\varphi_{\eps}(\cdot, t)) \, \dt \rightarrow\sigma  \int_0^T   |\nabla \tilde \chi(\cdot, t)|(\Omega) \, \dt.
      \end{equation}
    \item[v)] The triple $(\rho, \tilde \chi, \mathbf{u})$ is a distributional solution to the Verigin problem with phase transition in the sense of
    Definition~\ref{def:sharpsol}.
  \end{enumerate}
\end{theorem}

\begin{remark} 
Notice that a distributional solution to the Verigin problem with phase transition (cf. Definition~\ref{def:sharpsol}) is recovered only after a suitable choice of the representative for the limit in~\eqref{eq:convarphieps0}. Indeed, the convergence of $\varphi_\varepsilon$ is obtained only in $L^1_0(\Omega)=L^1(\Omega)/\operatorname{span}\{1\}$, and therefore determines the limit only up to spatially constant functions. Similarly as in~\cite{Schaetzle}, at ``regular" times $t\in R$ the limit $\chi(t)$ is uniquely identified as the characteristic function of a set of finite perimeter, whereas for $t\notin R$ it is only known to be spatially constant. Hence, we define $\tilde\chi=\chi$ on $R$ and suitably choose $\tilde\chi(t)\in\{0,1\}$ for $t\notin R$, so that the resulting triple $(\rho,\tilde\chi,u)$ is a distributional solution in the sense of Definition~\ref{def:sharpsol}.
\end{remark}
\begin{remark} 
We observe that the restrictions $d\in\{2,3\}$ and $m\geq 2$ in the hypothesis of Theorem~\ref{theo:sharplimit} are used 
to obtain the $H^2$-regularity of the phase field in Lemma~\ref{lemma:regularity}.
In particular, we apply Lemma~\ref{lemma:regularity} to justify 
the computation of the distributional divergence of the diffuse-interface stress tensor $\mathbb{T}_\eps$ in the first steps of the proof of Theorem~\ref{theo:sharplimit}$(i)$.
Moreover, the assumption $m\geq 2$ is also used in the proof of Theorem~\ref{theo:sharplimit}$(v)$ to pass to the limit in the coupling term when establishing the pointwise lower-semicontinuity estimate~\eqref{eq:pointwise-energy-liminf} for the total energy.
\end{remark}
\begin{remark} 
    The interfacial energy convergence~\eqref{eq:energyconv} is often imposed as an additional assumption in sharp-interface convergence results towards $BV$-type solutions of geometric flows, 
    both for Allen--Cahn approximations of mean curvature flow (cf., e.g.,~\cite{LauxSimon,HenselLauxcontact}) and for conserved phase-field models (cf., e.g.~\cite{KroemerLaux,LauxStinson}). Here, this convergence follows from the quasi-static structure of the Verigin problem with phase transition, through the almost-minimizing property~\eqref{eq:almostmin}.
    Moreover, in view of~\eqref{eq:energyconv} and the convergence of the coupling term, convergence of the total phase-field energy~\eqref{phasefieldenergy} would reduce to the convergence of the bulk energy term.
    In particular, if one could establish
    \begin{align*}
          \int_0^T \int_\Omega (\rho_\eps)^m  \, \dx \dt \to \int_0^T \int_\Omega (\rho)^m  \, \dx  \dt \quad \text{as } \eps \to 0,
    \end{align*}
    then the defect measure could be identified as $\mu = \rho^m \mathcal{L}^d \llcorner \Omega \otimes \mathcal{L}^1 \llcorner (0,T)$.
\end{remark}

\subsection{Compactness}
In this section, we prove the items $(i), (ii),$ and $(iii)$ of Theorem~\ref{theo:sharplimit}.

\subsubsection{Compactness of $\rho_\eps$} \label{sec:compsharp}
In this subsection we prove the strong convergence of $(\rho_\varepsilon)_{\varepsilon>0}$, namely the item $(i)$ of Theorem~\ref{theo:sharplimit}. 
First, we establish the $W_2$-H\"older continuity in time.
\begin{lemma}[Compactness in time]\label{lemma:comptimerhoeps}
    There exists a constant $C>0$ depending only on $T$ and $\mathcal{E}_\varepsilon(\rho_0, \varphi_0)$, such that, for any $\varepsilon>0$ and $0 \le t_0 <t_1\le T$, it holds $W_2(\rho_\varepsilon(\cdot,t_1),\rho_\varepsilon(\cdot,t_0)) \le C |t_1-t_0|^{1/2}$.
\end{lemma}

\begin{proof}
    The Benamou-Brenier formula reads as
    \begin{equation*}
        W_2^2(\rho(\cdot,t_1),\rho(\cdot,t_0)) = \inf_{\tilde \rho, \tilde{\mathbf{u}}} \Big \{\int_0^1\int_{\Omega} \tilde \rho(x,t) | \tilde{\mathbf{u}}(x,t)|^2 \, \dx \dt \Big \},
    \end{equation*}
    where the infimum is taken over all couples satisfying $\partial_t \tilde \rho + \text{div}(\tilde \rho \tilde{\mathbf{u}})=0$, $\tilde \rho(\cdot, 0) = \rho(\cdot, t_0)$, and $\tilde \rho(\cdot, 1) = \rho(\cdot, t_1)$.
    In particular, we can consider
    \[
    \tilde \rho_\varepsilon(x,t):= \rho_\varepsilon(x,t(t_1-t_0)+t_0), \quad \tilde{\mathbf{u}}_\varepsilon(x,t):= (t_1-t_0) \mathbf{u}_\varepsilon(x,t(t_1-t_0)+t_0),
    \]
    which satisfy~\eqref{eq:transport}. Then, from the Benamou-Brenier formula and~\eqref{eq:energydiss} it follows that
    \begin{align*}
        W_2^2(\rho_\varepsilon(\cdot,t_1),\rho_\varepsilon(\cdot,t_0)) &\le \int_0^1 \int_{\Omega} \tilde \rho_\varepsilon(x,t) |\tilde{\mathbf{u}}_\varepsilon(x,t)|^2 \, \dx \dt\\
        &= (t_1-t_0) \int_{t_0}^{t_1} \int_{\Omega}  \rho_\varepsilon(x,t) |{\mathbf{u}}_\varepsilon(x,t)|^2 \, \dx \dt\\
        &\le (t_1-t_0) \mathcal{E}_\varepsilon(\rho_0,\varphi_0),
    \end{align*}
    recall that, thanks to Remark~\ref{rmk:enpos}, we can assume that the energy $\mathcal{E}_\eps$ is positive up to adding a uniform constant.
\end{proof}

Next, we deduce space-compactness for $(\rho_\varepsilon)_{\varepsilon>0}$ and thus prove the desired result.
\begin{proof} [Proof of Theorem~\ref{theo:sharplimit}(i)]
First, we claim that
\begin{align}  
&\nabla(\rho_\varepsilon^m) \in L^1(\Omega \times (0,T)), \notag\\
   &     \tfrac{m}{m-1}\nabla (\rho_\eps^{m-1} ) = -   \mathbf{u}_\eps  +   \nabla \varphi_\eps \; \text{ a.e. in } \Omega \times (0,T).\label{eq:magica2}
\end{align}
Indeed, we recall the diffuse-interface stress tensor
$$
 \mathbb{T}_\eps = \Big(\frac{\eps}{2} |\nabla \varphi_\eps|^2 + \frac{1}{\eps} F(\varphi_\eps)  \Big) \operatorname{Id}  - \eps \nabla \varphi_\eps\otimes  \nabla \varphi_\eps,
$$
whose distributional divergence is given by
$$
\operatorname{div} (\mathbb{T}_\eps) = \Big( -\eps \Delta \varphi_\eps + \frac1\eps F'(\varphi_\eps)\Big) \nabla \varphi_\eps 
$$
and whose well-posedness is due to Lemma~\ref{lemma:regularity}.
Using~\eqref{eq:ACweak}, we obtain
$$
\operatorname{div} (\mathbb{T}_\eps) = \big( \rho_\eps - \lambda\big) \nabla \varphi_\eps 
$$
in the sense of distributions.
Hence, integrating by parts, we can rewrite~\eqref{eq:weaksol} as follows
 \begin{align*}
& - \int_0^T \int_\Omega \rho_\eps  \mathbf{u}_\eps \cdot \xi  \, \dx \dt  \\
&= 
\int_0^T \int_{\Omega}  \mathbb{T}_\eps :\nabla \xi  
   \, \mathrm{d}x \dt  -\int_0^T \int_{\Omega} (\rho_\eps^m - \lambda\varphi_\eps) \mathrm{div} \xi \,  \dx \dt \notag \\
   & = - \int_0^T \int_{\Omega}   \large( \rho_\eps - \lambda\large) \nabla \varphi_\eps  \cdot \xi  
   \, \mathrm{d}x \dt  -\int_0^T \int_{\Omega} (\rho_\eps^m - \lambda\varphi_\eps) \mathrm{div} \xi \,  \dx \dt \notag \\
& = - \int_0^T \int_{\Omega}   \large( \rho_\eps - \lambda\large) \nabla \varphi_\eps  \cdot \xi  
   \, \mathrm{d}x \dt  -\int_0^T \int_{\Omega}  (\rho_\eps^m \mathrm{div} \xi+ \lambda\nabla \varphi_\eps \cdot \xi) \,  \dx \dt\\
   & = \int_0^T \int_{\Omega}  (- \rho_\eps  \nabla \varphi_\eps  \cdot \xi  -  \rho_\eps^m  \mathrm{div} \xi)  \,  \dx \dt
   \notag
\end{align*}
    for all $\xi \in C^\infty_c(\overline \Omega \times [0, T); \mathbb{R}^d) $ with $\xi \cdot \nu_{\Omega} =0$ along $\partial \Omega \times (0,T)$. In particular, it holds
\begin{align} \label{eq:magica}
    \nabla (\rho_\eps^m )=- \rho_\eps  \mathbf{u}_\eps   + \rho_\eps  \nabla \varphi_\eps     
\end{align}
in the sense of distributions. 
Notice that, since $\nabla (\rho_\eps^m ) \in L^1(\Omega \times (0,T))$ by comparison in~\eqref{eq:magica} due to the regularity of the solutions in the sense of Definition~\ref{def:weaksol}, then~\eqref{eq:magica} holds almost everywhere in $\Omega \times (0,T)$. 
Moreover, by assumption $\rho_\eps \geq c >0$,~\eqref{eq:magica} yields the claim~\eqref{eq:magica2}.

We define $$\phi_\eps := j(\rho_\eps) - \varphi_\eps ,$$ where $j(\rho):=\frac{m}{m-1}\rho^{m-1}$. Then, by~\eqref{eq:magica2}, we obtain that $\nabla \phi_\eps = -\mathbf{u}_\eps$ almost everywhere in $\Omega \times (0,T)$.
Moreover, let $a_\eps:= \phi_\eps + \psi_\eps,$ where $\psi_\eps$ is given by Lemma~\ref{lemma:Schaetzle}. We have that $a_\eps \in L^\infty(0,T;L^1(\Omega))$, $|\nabla a_\eps(t)|=|-\mathbf{u}_\eps(t) + \nabla \psi_\eps(t)|$,  and $\int_0^T\| a_\eps(t)\|_{BV(\Omega)} \, \dt\le C$. 
Observe that $a_\eps-j(\rho_\eps)=\psi_\eps -\varphi_\eps$, therefore Lemma~\ref{lemma:Schaetzle} yields
\begin{align}
    a_\eps-j(\rho_\eps) \to 0 \quad  &\text{in } L^\infty(0,T, L^1(\Omega)), \label{eq:convjgrho} \\
    a_\eps-j(\rho_\eps) \to 0 \quad &\text{a.e. } \Omega \times (0,T). \notag
\end{align}
Defining
\[
g_\eps:=j|_{[0,+\infty)}^{-1}((a_\eps)_{+}),
\]
where $x_+:=\max\{x,0\}$, we obtain $\int_0^T\| j(g_\eps(t))\|_{BV(\Omega)} \, \dt\le C$.  Since the function $x \mapsto x_+$ is $1$-Lipschitz and $\rho_\varepsilon \ge c>0$ by assumption, we deduce 
\begin{align*}
    j(g_\eps)-j(\rho_\eps) \to 0   \text{ in } L^\infty(0,T, L^1(\Omega)), \quad 
    j(g_\eps)-j(\rho_\eps) \to 0  \text{ a.e. } \Omega \times (0,T). \notag
\end{align*}
Consequently, we obtain
\begin{align}
     g_\eps-\rho_\eps \to 0 \quad &\text{a.e. } \Omega \times (0,T) , \notag \\
     g_\eps-\rho_\eps \to 0 \quad &\text{in } 
     L^{1}(\Omega \times (0,T)). \label{eq:convgrho}
\end{align}


We introduce the lower semicontinuous functional $\mathcal{F}: L^1(\Omega) \to [0,+\infty]$ defined as
        \[
    \mathcal{F}(u) :=  \|j(u)\|_{BV(\Omega)}.
    \]
    We observe that, for every $C>0$, the set $\mathcal{V}:=\{
v \in L^1(\Omega;[0,+\infty]): \mathcal{F}(v) \le C\} 
$
is compactly embedded in $L^1(\Omega)$ under the strong topology. 
To prove this, consider a sequence $(v_n)_n \subset \mathcal{V}$ and notice that, up to a subsequence, $j(v_n)$ converges to some $ w$ strongly in $ L^1(\Omega)$ and almost everywhere in $\Omega$. The strict monotonicity of $j$ guarantees its invertibility, implying
 $
 v_n = (j)^{-1}(j(v_n)) \to (j)^{-1}(w)$ a.e. in $\Omega$.
The strong convergence in $L^1(\Omega)$ then follows.

Consider a subsequence $\varepsilon_k\downarrow0$ and denote by $g_{k}=g_{\varepsilon_k}$ and $\rho_{k}=\rho_{\varepsilon_k}$. Let $\tilde W_2$ denote the generalized 2-Wasserstein distance (cf.~\cite[Definition 1]{Piccoli-Rossi}). For every $k_0 \in \mathbb{N}$ we have 
\begin{align*}
    &\lim_{\tau \to 0} \sup_{k \in \mathbb{N}} \int_0^{T-\tau} \tilde W_2(g_k(t+\tau),g_k(t))\, \dt \\
    &\le
    \lim_{\tau \to 0} \sup_{k \le k_0} \int_0^{T-\tau} \tilde W_2(g_k(t+\tau),g_k(t))\, \dt + \lim_{\tau \to 0} \sup_{k > k_0} \int_0^{T-\tau} \tilde W_2(g_k(t+\tau),g_k(t))\, \dt \\
    &= \lim_{\tau \to 0} \sup_{k > k_0} \int_0^{T-\tau} \tilde W_2(g_k(t+\tau),g_k(t))\, \dt,
\end{align*}
where the last equality follows from the continuity of time translations of $g_{k}$ in $L^1(0,T;L^1(\Omega))$ for fixed $k$.
Furthermore, we have 
\begin{align*}
    &\lim_{\tau \to 0} \lim_{k_0 \to \infty}\sup_{k> k_0} \int_0^{T-\tau} \tilde W_2(g_k(t+\tau),g_k(t))\, \dt \\
    &\le \lim_{\tau \to 0} \lim_{k_0 \to \infty}\sup_{k> k_0} \bigg(\int_0^{T-\tau} \|g_k(t+\tau) - \rho_k(t+\tau)\|_{L^1(\Omega)}+\|g_k(t) - \rho_k(t)\|_{L^1(\Omega)} \, \dt \\
    & \quad + \int_0^{T-\tau} W_2(\rho_k(t+\tau),\rho_k(t)) \, \dt \bigg)
    =0,
\end{align*}
where the terms on the right-hand side yield zero due to~\eqref{eq:convgrho} and Lemma~\ref{lemma:comptimerhoeps}.

Notice that all the assumptions of Theorem~\ref{theo:RossiSava} are satisfied with $d=\tilde W_2$ (generalized 2-Wasserstein distance), $\mathcal{U}:=(g_{k})_{k \in \mathbb{N}}$, in particular $\mathcal{U}$ is tight with respect to $\mathcal{F}$. 
Applying Theorem~\ref{theo:RossiSava} and using the strong convergence~\eqref{eq:convgrho}, we obtain that $\mathcal{U}$ is relatively compact in $\mathcal{M}(0,T; L^1(\Omega))$, yielding the strong convergence $g_\varepsilon(\cdot,t) \rightarrow \rho(\cdot,t)$ in $L^1(\Omega)$ for almost all $t \in (0,T)$. Moreover, by an application of the dominated convergence theorem, we obtain
\begin{align}\label{eq:convg_eps}
    g_\varepsilon \rightarrow \rho \quad \text{in }L^1(\Omega \times (0,T)).
\end{align}
In particular, we have that 
\begin{equation*}
    \rho_\eps \to \rho \quad \text{in }L^1(\Omega \times (0,T)),
\end{equation*}
due to~\eqref{eq:convgrho}. Finally, the bound~\eqref{eq:energydiss} yields 
   \begin{equation} 
   \rho_{\eps}\rightharpoonup\rho\quad\text{weakly in } L^m(\Omega \times (0,T)), \quad  (\rho_{\eps})^m \overset{*}{\rightharpoonup} \mu \quad\text{weakly-* in measure}
   \label{eq:weakconvrhoeps}
   \end{equation}
   for some $\mu \in \mathcal{M}_b(\Omega \times (0,T))$ such that $\mu \ge \rho^m \mathcal{L}^d \llcorner \Omega \otimes \mathcal{L}^1 \llcorner (0,T)$,
and the desired strong convergence~\eqref{eq:convrhoeps}.

\end{proof}

\subsubsection{Compactness of $\varphi_\eps$}\label{sec:compvarphieps}
In this subsection we prove the strong convergence of $(\varphi_\varepsilon)_{\varepsilon>0}$, which is stated in Theorem~\ref{theo:sharplimit}$(ii)$.

\begin{proof} [Proof of Theorem~\ref{theo:sharplimit}(ii)]
First, notice that the bound~\eqref{eq:energydiss} yields 
   \begin{equation} \varphi_{\eps}\rightharpoonup\chi\qquad\text{weakly in } L^4(\Omega \times (0,T)).\label{eq:weakconvvarphieps} \end{equation}

Recall that $(\psi_\eps)_{\eps>0}$ is the family given by Lemma~\ref{lemma:Schaetzle}, that $a_\eps= \phi_\eps +\psi_\eps$, and that $\nabla \phi_\eps = -\mathbf{u}_\eps$. Let $\tau \in (0,T)$ and, for any function $f:(0,T) \to \mathbb{R}$, let $f_\tau := f(\cdot+ \tau).$
Applying Lemma~\ref{lemma:Luckhaus} with  $\psi=\psi_\eps-\psi_{\eps,\tau}$ and $\phi=\phi_\eps-\phi_{\eps,\tau}$, we obtain
\begin{align*}
			&  \int_0^{T-\tau}\| \psi_\eps - \psi_{\eps,\tau}\|_{L_0^1(\Omega)} \, \dt\notag\\
            &\leq C  \int_0^{T-\tau}\|a_\eps- a_{\eps,\tau}\|_{L^{1}(\Omega)} \,\dt  \\
            & \quad +   C\int_0^{T-\tau}\|a_\eps- a_{\eps,\tau}\|_{L^{1}(\Omega)}^{\frac{1}{2}}\bigg(  \int_{\Omega} |\nabla \phi_\eps|^2  \,\dx +  \int_{\Omega} |\nabla  \phi_{\eps,\tau}|^2  \,\dx\bigg)^{\frac{1}{2}} \dt.
\end{align*}
Combining the above inequality with the convergences~\eqref{eq:convrhoeps} and~\eqref{eq:convjgrho} and the energy dissipation bound~\eqref{eq:energydiss}, and recalling that  $|\nabla \psi_\eps(t)|(\Omega)$ is uniformly bounded for almost all $t \in (0,T)$ by Lemma~\ref{lemma:Schaetzle}, we can apply the Frech\'et-Kolmogorov theorem to conclude that $(\psi_\eps)_{\eps>0}$ is relatively compact in $L^1(0,T;L_0^1(\Omega)),$ 
in particular $(\psi_\eps(t))_{\eps>0}$ is relatively compact in $L_0^1(\Omega)$ for almost all $t \in (0,T)$.
Notice that, if there are two subsequences
$\varepsilon_j^{(i)}\to 0$, $i=1,2$, such that
$
    \psi_{\varepsilon_j^{(i)}}(t)
    \to \psi^{(i)}
    \text{ in }L^1(\Omega)
$, then
$
    \psi^{(1)}=\psi^{(2)}
    \text{ in }L_0^1(\Omega),
$
meaning that
$
    \psi^{(1)}-\psi^{(2)}
    = \text{constant in }\Omega,
$
and therefore, since $\psi^{(i)}\in BV(\Omega;\{0,1\})$, either
$
    \psi^{(1)}=\psi^{(2)}
$
or
$
\{\psi^{(1)},\psi^{(2)}\}=\{0,1\}.
$
Using Lemma~\ref{lemma:Schaetzle}, we then obtain~\eqref{eq:convarphieps0}, i.e., up to a non-relabeled subsequence,
 \begin{equation*}
\varphi_{\varepsilon} \to\chi 
\quad\text{ strongly in } L^1(0,T;L^1_0(\Omega)).
  \end{equation*}

Let $R$ denote the set of such ``regular" times of $t \in [0,T]$ such that $(\psi_\eps(t))_{\eps>0}$ converges strongly in $L^1(\Omega)$. Observe that $R$ is measurable. Hence, for $t \in R$ also $\varphi_\eps(\cdot, t)\to \chi(\cdot, t)$ strongly in $L^1(\Omega)$ and
\begin{equation}\label{eq:convvarphitildechinew}
    \varphi_\varepsilon \to \chi \text{ in $L^1(R\times\Omega)$}.
\end{equation}
Moreover, for times $t\in R$, we can show that $\chi(\cdot, t)$ is the characteristic function of a set of finite perimeter. 
Indeed, since by~\eqref{eq:energydiss} and the well-preparedness of the initial data we have
 \begin{equation*}
 \sup_{\eps>0 } \esssup_{t\in [0,T]} \mathcal{E}_{\textnormal{int}, \eps}(\varphi_\eps{(\cdot, t)}) \leq C<\infty,
 \end{equation*}
 Fatou's lemma gives
 $\int_{\Omega} F (\chi(x,t))\, \dx   = 0$, which implies that $\chi(x,t) = \chi_{E}(x,t)$ for some set $E\subset \Omega\times (0,T).$ From the coarea formula we deduce that $E(t)$ is a set of finite perimeter in $\Omega$, where $E(t)$ is identified by $\chi_{E(t)}: = \chi_{E}(\cdot, t)$. 

 On the other hand, it can be shown that the weak limit $\chi$ in~\eqref{eq:weakconvvarphieps} satisfies $\chi(t)\equiv \text{const}$ whenever $t \notin R$. Similarly as in~\cite{Schaetzle}, using the weak convergence~\eqref{eq:weakconvvarphieps} together with Lemma~\ref{lemma:Schaetzle}, we obtain that for any measurable $I\subset [0,T]\setminus R$ and
$\xi\in C_c^1(\Omega;\mathbb{R}^d)$
\[
    \int_I\int_\Omega \chi\,\operatorname{div}\xi \, \dx \dt 
    =
    \lim_{\varepsilon\to 0}
    \int_I\int_\Omega
    \varphi_\varepsilon\,\operatorname{div}\xi \, \dx \dt 
    =
    \lim_{\varepsilon\to 0}
    \int_I\int_\Omega
    \psi_\varepsilon\,\operatorname{div}\xi \, \dx \dt 
    =
    0.
\]
The last equality follows by dominated convergence theorem from the fact that 
$
    \int_\Omega
    \psi_\varepsilon(t)\,\operatorname{div}\xi \, \dx
$
is bounded, and for every subsequence $\varepsilon\to 0$ there is a
subsequence $\varepsilon_j\to 0$ such that
$\psi_{\varepsilon_j}(t)\to c \in \{0, 1\}$ in $L^1(\Omega)$, hence
\[
    \lim_{j\to\infty}
    \int_\Omega
    \psi_{\varepsilon_j}(t)\,\operatorname{div}\xi \, \dx 
    =
    c\int_\Omega \operatorname{div}\xi \, \dx 
    =
    0.
\]
Since $\Omega$ is connected, $\chi(t)$ is constant for almost all
$t\notin R$.

\end{proof}

\subsubsection{Convergence of the velocity field}
In this subsection we prove the weighted convergence of the velocity field in the sense of Theorem~\ref{theo:sharplimit}$(iii)$.

\begin{proof} [Proof of Theorem~\ref{theo:sharplimit}(iii)]
Defining $\mathbf{v}_\varepsilon := \sqrt{\rho_\varepsilon} \mathbf{u}_\varepsilon $, it follows from~\eqref{eq:energydiss} that (up to a subsequence)
\begin{equation} \label{eq:limitveps}
\sqrt{\rho_\varepsilon} \mathbf{u}_\varepsilon = \mathbf{v}_\varepsilon \rightharpoonup \mathbf{v} \quad \text{weakly in } L^2(\Omega \times (0,T)).
\end{equation}
Since $\rho_\eps \geq c >0$ and hence $\rho \geq c >0$ by almost-everywhere convergence following from~\eqref{eq:convrhoeps}, we can define $\mathbf{u}:=\mathbf{v}/\sqrt{\rho}$ and obtain the desired convergence~\eqref{eq:convweigthedueps}.
In particular, we have
\begin{equation}\label{eq:lscdissip}
\int_0^T \int_\Omega \rho|\mathbf{u}|^2 \, \dx \dt\le \liminf_{\eps \to 0} \int_0^T\int_{\Omega} \rho_\varepsilon|\mathbf{u}_\varepsilon|^2 \, \dx \dt.
\end{equation}
Moreover, notice that the convergences~\eqref{eq:convrhoeps} and~\eqref{eq:limitveps} imply
\begin{equation} \label{eq:limitrhoueps}
 \rho_\varepsilon\mathbf{u}_\varepsilon \rightharpoonup \rho \mathbf{u} \quad \text{weakly in } L^{\frac{2p}{p+1}}(\Omega \times (0,T)), \quad 1 \leq p <m.
\end{equation}
\end{proof}

\subsection{Interfacial energy convergence}\label{sec:energyconvsharp}
In this section, we prove the convergence of the time-integrated interfacial phase-field energy, namely Theorem~\ref{theo:sharplimit}$(iv)$.

\begin{proof} [Proof of Theorem~\ref{theo:sharplimit}(iv)]
First, we recall that $(\rho_\varepsilon,\varphi_\varepsilon)$ satisfies the almost minimizing property~\eqref{eq:almostmin}, namely
\begin{align*}
    & \int_0^T\int_{\Omega} \Big ( \frac{\varepsilon}{2}|\nabla \varphi_\eps|^2 + \frac{1}{\varepsilon} F(\varphi_\eps) \Big ) \, \dx \dt +\int_0^T \int_{\Omega} \varphi_\eps  (\lambda-\rho_\eps)  \, \dx \dt \notag \\
    &\le \int_0^T\int_{\Omega} \Big ( \frac{\varepsilon}{2}|\nabla \tilde \varphi|^2 + \frac{1}{\varepsilon} F(\tilde \varphi) \Big ) \, \dx \dt
    +  \int_0^T\int_{\Omega} \tilde \varphi (\lambda- \rho_\eps)  \, \dx\dt +\varepsilon
\end{align*}
for every $\tilde \varphi \in L^1(0,T;H^1(\Omega))$.

Let $\tilde \chi(t)=\chi(t)$ if $t \in R$ and $\tilde \chi(t) \equiv c \in \{0,1\}$ if $t \notin R$.
Recall the convergences~\eqref{eq:convarphieps0} and~\eqref{eq:convvarphitildechinew}.
Since both $\chi(t)$ and $\tilde \chi(t)$ are constant when $t \notin R$ and they coincide when $t \in R$, we obtain that $|\nabla  \chi| = |\nabla  \tilde \chi|$ almost everywhere in $\Omega \times (0,T)$. Hence, Fatou's lemma together with the Modica--Mortola liminf inequality (cf.~\cite{ModicaMortola}) yields
\begin{align*} 
    &\sigma\int_0^T\int_{\Omega} 1\,\mathrm{d}|\nabla \tilde \chi| \dt=  \sigma \int_{R} \int_{\Omega} 1\,\mathrm{d}|\nabla \tilde \chi| \dt\\
    & 
    \le \liminf_{\varepsilon\to 0} \int_R \int_{\Omega}\Big (\frac{\varepsilon}{2}|\nabla \varphi_\eps|^2 + \frac{1}{\varepsilon} F(\varphi_\eps) \Big ) \dx \dt\\
    &
    \le \liminf_{\varepsilon\to 0} \int_0^T \int_{\Omega}\Big (\frac{\varepsilon}{2}|\nabla \varphi_\eps|^2 + \frac{1}{\varepsilon} F(\varphi_\eps) \Big ) \dx \dt,
\end{align*}
where $\sigma := \int_0^1\sqrt{2F(s)}\,\mathrm{d}s$, whereas the convergences~\eqref{eq:convrhoeps} and~\eqref{eq:weakconvvarphieps} give
\begin{align*}
    \int_0^T\int_{\Omega} \chi(\lambda -\rho)\, \dx \dt
    = \lim_{\varepsilon\to 0} \int_0^T \int_{\Omega}\varphi_\varepsilon(\lambda -\rho_\varepsilon)\, \dx \dt .
\end{align*}

Let $\eta \in L^1(0,T;BV(\Omega;\{0,1\}))$, there exists $(\eta_\varepsilon)_\varepsilon \subset L^1(0,T;H^1(\Omega))$ such that $\eta_\varepsilon \to \eta$ strongly in $L^1(\Omega \times (0,T))$ and $\int_0^T\int_{\Omega} \left( \frac{\varepsilon}{2}|\nabla \eta_\eps|^2 + \frac{1}{\varepsilon} F(\eta_\eps) \right) \, \dx \dt \to \sigma \int_0^T\int_{\Omega} 1\,\mathrm{d}|\nabla \eta| \dt$ (cf.~\cite{ModicaMortola}). Additionally, $\eta_\varepsilon \rightharpoonup \eta$ weakly in $L^4(\Omega \times (0,T))$.
Hence, using the convergence~\eqref{eq:convrhoeps} together with the almost-minimizing property~\eqref{eq:almostmin}, we obtain
\begin{align*}
    & \limsup_{\varepsilon\to 0} \int_0^T \int_{\Omega}\Big (\frac{\varepsilon}{2}|\nabla \varphi_\eps|^2 + \frac{1}{\varepsilon} F(\varphi_\eps) +\varphi_\varepsilon (\lambda -\rho_\varepsilon)\Big ) \dx \dt \\
    &\le \limsup_{\varepsilon\to 0} \int_0^T \int_{\Omega}\Big (\frac{\varepsilon}{2}|\nabla \eta_\eps|^2 + \frac{1}{\varepsilon} F(\eta_\eps) +\eta_\varepsilon (\lambda -\rho_\varepsilon)\Big ) \dx \dt+\varepsilon\\
    &=\int_0^T \Big(\sigma \int_{\Omega} 1\,\mathrm{d}|\nabla \eta|+
    \int_{\Omega} \eta(\lambda-\rho) \, \dx \Big ) \dt.
\end{align*}
In particular, considering $\eta = \tilde \chi$ with the particular choice $\tilde \chi(t)=\chi(t)$ for $t \in R$ and
\begin{equation}\label{eq:tildechi}
\tilde \chi (t) \equiv \begin{cases}
    0 & \text{if } \lambda |\Omega|-1>0\\
    1 & \text{if } \lambda |\Omega|-1<0 \\
    c \in \{0,1\} & \text{if } \lambda |\Omega|-1=0
\end{cases}, \quad \text{ for }  t \notin R,
\end{equation} we have
\begin{align*}
    & \limsup_{\varepsilon\to 0} \int_0^T \int_{\Omega}\Big (\frac{\varepsilon}{2}|\nabla \varphi_\eps|^2 + \frac{1}{\varepsilon} F(\varphi_\eps) +\varphi_\varepsilon (\lambda -\rho_\varepsilon)\Big ) \dx \dt \\
    &\le \int_0^T \Big(\sigma \int_{\Omega} 1\,\mathrm{d}|\nabla \tilde \chi|+
    \int_{\Omega} \tilde \chi(\lambda-\rho) \, \dx \Big ) \dt\\
    &\le \int_0^T \Big(\sigma \int_{\Omega} 1\,\mathrm{d}|\nabla \tilde \chi|+
    \int_{\Omega}  \chi(\lambda-\rho) \, \dx \Big ) \dt.
\end{align*}

From the inequalities above one can deduce
\begin{equation*}
    \sigma\int_0^T \int_{\Omega} 1 \, \mathrm{d}|\nabla \tilde \chi|\le \lim_{\varepsilon \to 0}\int_0^T \int_{\Omega} \frac{\varepsilon}{2}|\nabla \varphi_\eps|^2 + \frac{1}{\varepsilon} F(\varphi_\eps) \,  \dx \dt \le     \sigma\int_0^T \int_{\Omega} 1 \, \mathrm{d}|\nabla  \tilde \chi|,
\end{equation*}
hence we conclude
\begin{equation*}
    \lim_{\varepsilon \to 0}\int_0^T \int_{\Omega} \frac{\varepsilon}{2}|\nabla \varphi_\eps|^2 + \frac{1}{\varepsilon} F(\varphi_\eps) \,  \dx \dt =     \sigma\int_0^T \int_{\Omega} 1 \, \mathrm{d}|\nabla \tilde \chi|.
\end{equation*}
\end{proof}

\subsection{Convergence to the Verigin problem with phase transition}\label{sec:limitsharp}
In this section, we prove the item $(v)$ of Theorem~\ref{theo:sharplimit}.

\begin{proof} [Proof of Theorem~\ref{theo:sharplimit}(v)]
We divide the proof into several steps.

\textit{Step 1 (Transport equation).} The convergences~\eqref{eq:convrhoeps0},~\eqref{eq:convrhoeps} and~\eqref{eq:limitrhoueps} suffice to pass to the limit as $\eps \to 0$ in~\eqref{eq:transport}, obtaining~\eqref{eq:transportsharp}.

\textit{Step 2 (Young-Laplace law).} The left-hand side of~\eqref{eq:weaksol} passes to the desired limit due to~\eqref{eq:limitrhoueps}.

To pass to the limit in the second term on the right-hand side of~\eqref{eq:weaksol} we need to show that
\begin{equation*}
    \mathbb{T}_\eps \, \dx \dt \overset{\ast}{\rightharpoonup} \sigma (\operatorname{Id} - \nu \otimes \nu ) \, \mathrm{d} |\nabla \tilde \chi| \dt 
    \quad \text{ weakly-* as measures,}
\end{equation*}
where we recall that 
$$
 \mathbb{T}_\eps = \Big(\frac{\eps}{2} |\nabla \varphi_\eps|^2 + \frac{1}{\eps} F(\varphi_\eps)  \Big) \operatorname{Id}  - \eps \nabla \varphi_\eps\otimes  \nabla \varphi_\eps
$$
and that $\nu $ denotes the Radon-Nikodym derivative $-\frac{ \mathrm{d} \nabla \tilde\chi}{\mathrm{d} |\nabla \tilde\chi|}$.

We first define $$
\Psi_\eps := \Psi (\varphi_\eps) = \int_0^{\varphi_\eps} \sqrt{2F(s)} \, \mathrm{d} s,$$
then we test $\frac{\eps}{2} |\nabla \varphi_\eps|^2 + \frac{1}{\eps} F(\varphi_\eps)  $ with $\eta $ such that $\eta \in C_c(\Omega \times (0,T))$ and $\eta \ge 0$, obtaining 
\begin{align*}
    &\liminf_{\eps \to 0} \int_0^T \int_\Omega \eta\Big(\frac{\eps}{2} |\nabla \varphi_\eps|^2 + \frac{1}{\eps} F(\varphi_\eps)  \Big)  \, \dx \dt \\
    &\geq
 \liminf_{\eps \to 0} \int_0^T \int_\Omega \eta |\nabla \Psi_\eps|  \, \dx \dt \geq
 \liminf_{\eps \to 0} \int_R \int_\Omega \eta |\nabla \Psi_\eps|  \, \dx \dt\\
 &\geq  
    \sigma \int_R \int_\Omega \eta \, \mathrm{d} |\nabla \tilde \chi| \dt =  
    \sigma \int_0^T \int_\Omega \eta \, \mathrm{d} |\nabla \tilde \chi| \dt,
\end{align*}
due to Fatou's lemma, the lower semicontinuity of the total variation measure, and the strong convergence $\Psi_\eps \to \Psi(\tilde \chi) = \sigma \tilde \chi$ in $L^1(\Omega)$ for almost every $t \in R$ following from~\eqref{eq:convvarphitildechinew}. 
Similarly, we can test with $\|\eta\|_{\infty}-\eta $ and obtain the same inequalities with $\|\eta\|_{\infty}-\eta $ (instead of $\eta$).
These inequalities combined with~\eqref{eq:energyconv} yield 
\begin{equation*}
\int_0^T \int_\Omega  \eta\Big(\frac{\eps}{2} |\nabla \varphi_\eps|^2 + \frac{1}{\eps} F(\varphi_\eps)  \Big) \, \dx \dt \to \sigma \int_0^T \int_\Omega \eta \, \mathrm{d} |\nabla \tilde \chi| \dt 
\end{equation*}
and 
\begin{equation*}
\int_0^T \int_\Omega  \eta|\nabla \Psi_\eps|  \, \dx \dt \to \sigma \int_0^T \int_\Omega \eta \, \mathrm{d} |\nabla \tilde \chi| \dt 
\end{equation*}
for every $\eta \in C_c(\Omega \times (0,T))$.
In particular, it holds 
\begin{equation*}
   \Big(\frac{\eps}{2} |\nabla \varphi_\eps|^2 + \frac{1}{\eps} F(\varphi_\eps)  \Big) \, \dx \dt  \overset{\ast}{\rightharpoonup} \sigma  \, \mathrm{d} |\nabla \tilde \chi|  \,\dt 
    \quad \text{ weakly-* as measures}
\end{equation*}
and
\begin{equation*}
   |\nabla \Psi_\eps| \, \dx \dt   \overset{\ast}{\rightharpoonup} \sigma  \, \mathrm{d} |\nabla \tilde \chi| \, \dt
    \quad \text{ weakly-* as measures.}
\end{equation*}
Furthermore, one can deduce the equipartition of energy, namely
\begin{equation*}
  \int_0^T \int_\Omega \Big(\sqrt{\eps} |\nabla \varphi_\eps| - \frac{1}{\sqrt{\eps}} \sqrt{2F(\varphi_\eps) } \Big)^2  \dx \dt  \to 0,
\end{equation*}
and 
\begin{equation*} 
   \eps |\nabla \varphi_\eps|^2 \, \dx \dt - |\nabla \Psi_\eps| \, \dx \dt  \overset{\ast}{\rightharpoonup} 0
    \quad \text{ weakly-* as measures}.
\end{equation*}
For further details on the proof of these convergences, we refer, for instance, to~\cite{LuckhausModica}.

Recalling~\eqref{eq:energyconv}, an application of Reshetnyak's continuity theorem~\cite[Theorem 2.39]{AFP} to the sequence of measures $\nabla \Psi_\eps \, \dx  \dt $ with the
1-homogeneous function $f(x, p) := \frac{p \otimes p}{|p|}: \nabla \xi(x) $ yields
\begin{align*}
    \lim_{\eps \to 0} \int_0^T\int_\Omega \frac{\nabla \varphi_\eps}{|\nabla \varphi_\eps|} \otimes \frac{\nabla \varphi_\eps}{|\nabla \varphi_\eps|}     : \nabla \xi |\nabla \Psi_\eps | \, \dx  \dt 
    = \sigma \int_0^T\int_\Omega \nu\otimes \nu : \nabla \xi \,  \mathrm{d} |\nabla \tilde \chi | \dt,
\end{align*}
whence we obtain
\begin{align*}
    \lim_{\eps \to 0} \int_0^T \int_\Omega \eps {\nabla \varphi_\eps} \otimes {\nabla \varphi_\eps}    : \nabla \xi \, \dx  \dt 
    = \sigma \int_0^T \int_\Omega \nu\otimes \nu : \nabla \xi \, \mathrm{d} |\nabla \tilde \chi| \dt.
\end{align*}

Finally, the last terms on the right-hand side of~\eqref{eq:weaksol} pass to the desired limit due to~\eqref{eq:weakconvrhoeps} and~\eqref{eq:weakconvvarphieps}. Recalling that  the limit function $\chi$ is constant when $t \notin R$, hence we have $\int_{\Omega} \lambda \chi(t) \mathrm{div} \xi(t) \, \dx = 0 = \int_{\Omega} \lambda \tilde \chi(t) \mathrm{div} \xi(t) \, \dx$.

\emph{Step 3 (Energy dissipation law).}
We first establish the pointwise lower-semicontinuity estimate
\begin{equation}
\label{eq:pointwise-energy-liminf}
\mathcal{E}(\rho(\tau),\tilde\chi(\tau))
\leq
\liminf_{\varepsilon\to0}
\mathcal{E}_\varepsilon(\rho_\varepsilon(\tau),\varphi_\varepsilon(\tau))
\end{equation}
for almost every $\tau\in(0,T)$.

By the strong convergence~\eqref{eq:convrhoeps} and the energy dissipation~\eqref{eq:energydiss}, we have
\begin{equation}
\label{eq:slice-weak-rho}
\rho_\varepsilon(\tau)\to\rho(\tau)
\quad \text{strongly in }L^1(\Omega), \quad \rho_\varepsilon(\tau)\rightharpoonup\rho(\tau)
\quad \text{weakly in }L^m(\Omega),
\end{equation}
for almost every $\tau\in(0,T)$.
Let $\tau\in R$ be such that~\eqref{eq:slice-weak-rho} holds. The definition of $R$ from Section~\ref{sec:compvarphieps} together with the uniform energy bound yield
$$
\varphi_\varepsilon(\tau)\to
\tilde\chi(\tau)
\quad \text{strongly in }L^2(\Omega).
$$
Hence, since $m\geq2$, we obtain
$$
\int_\Omega
\varphi_\varepsilon(\tau)
(\lambda-\rho_\varepsilon(\tau))\, \dx
\to
\int_\Omega
\tilde\chi(\tau)
(\lambda-\rho(\tau))\, \dx.
$$
Furthermore, the weak lower semicontinuity from~\eqref{eq:slice-weak-rho} gives
$$
\int_\Omega \rho(\tau)^m\,\dx
\leq
\liminf_{\varepsilon\to0}
\int_\Omega \rho_\varepsilon(\tau)^m\, \dx,
$$
whereas the Modica--Mortola lower bound (cf.~\cite{ModicaMortola}) gives
$$
\sigma |\nabla\tilde\chi(\tau)|(\Omega)
\leq
\liminf_{\varepsilon\to0}
E_{\mathrm{int},\varepsilon}
\bigl(\varphi_\varepsilon(\tau)\bigr).
$$
Therefore,~\eqref{eq:pointwise-energy-liminf} holds for almost every
$\tau\in R$.

Consider now $\tau\notin R$ and a subsequence realizing the liminf in~\eqref{eq:pointwise-energy-liminf}. By the compactness argument used in Section~\ref{sec:compvarphieps}, we may extract a further subsequence such that
$$
\varphi_{\varepsilon_k}(\tau)\to c
\qquad\text{strongly in }L^2(\Omega)
$$
for some $c\in\{0,1\}$. Hence, using~\eqref{eq:slice-weak-rho}, we obtain
$$
\int_\Omega
\varphi_{\varepsilon_k}(\tau)
(\lambda-\rho_{\varepsilon_k}(\tau))\,\dx
\rightarrow
c(\lambda|\Omega|-1).
$$
Recalling the definition of $\tilde \chi$ on $(0,T)\setminus R$ given by~\eqref{eq:tildechi}, since $|\nabla \tilde \chi(\tau)|(\Omega)=0$ for $\tau\notin R$, we have 
$$
\begin{aligned}
\liminf_{\varepsilon\to0}
\mathcal{E}_\varepsilon(\rho_\varepsilon(\tau),\varphi_\varepsilon(\tau))
&\geq
\frac{1}{m-1}\int_\Omega\rho(\tau)^m\, \dx
+c\bigl(\lambda|\Omega|-1\bigr)
\\
&\ge
\mathcal{E}(\rho(\tau),\tilde\chi(\tau)),
\end{aligned}
$$
This proves~\eqref{eq:pointwise-energy-liminf} for almost every
$\tau\in(0,T)$.

Finally, from the weak convergence~\eqref{eq:convweigthedueps} it follows that
$$
\int_0^\tau\int_\Omega
\rho|\mathbf{u}|^2\,\dx\dt
\leq
\liminf_{\varepsilon\to0}
\int_0^\tau\int_\Omega
\rho_\varepsilon|\mathbf{u}_\varepsilon|^2\,\dx\dt
$$
for every $\tau\in(0,T)$.
Combining this inequality with~\eqref{eq:pointwise-energy-liminf} and the
energy dissipation inequality~\eqref{eq:energydiss}, we can conclude that, for almost every
$\tau\in(0,T)$,
$$
\begin{aligned}
\mathcal{E}(\rho(\tau),\tilde\chi(\tau))
+\int_0^\tau\!\!\int_\Omega\rho|\mathbf{u}|^2\,\dx\dt
&\leq
\liminf_{\varepsilon\to0}
\left(
\mathcal{E}_\varepsilon(\rho_\varepsilon(\tau),\varphi_\varepsilon(\tau))
+
\int_0^\tau\!\!\int_\Omega
\rho_\varepsilon|\mathbf{u}_\varepsilon|^2\,dxdt
\right)
\\
&\leq
\lim_{\varepsilon\to0}
\mathcal{E}_\varepsilon(\rho_{\varepsilon,0},\varphi_{\varepsilon,0})
=
\mathcal{E}(\rho_0,\chi_{E_0}),
\end{aligned}
$$
where the last equality follows from the well-preparedness of the initial data
\eqref{eq:convrhoeps0}--\eqref{eq:conveneps0}.
\end{proof}

\appendix 

\section{Compactness in measure} \label{appendixA}
In this section, we state a variant of the Aubin--Lions
Lemma, which can be found in~\cite[Theorem 2]{Rossi-Savare}.

Let $B$ be a separable Banach space. We say that $\mathcal{F}:B \to [0,+\infty]$ is a 
\emph{coercive integrand} if
\[
\{v \in B : \mathcal{F}(v) \le C\} \text{ is compact for every } C \ge 0.
\] 

We denote by $\mathcal{M}(0,T;B)$ the space of Bochner-measurable $B$-valued functions with the topology induced by the convergence in measure, i.e. $(u_n)_{n \in \mathbb{N}} \subset \mathcal{M}(0,T;B)$ converges in measure to $u \in \mathcal{M}(0,T;B)$ as $n \to \infty$ if
\[
\lim_{n \to \infty} |\{t \in (0,T) :\|u_n(t)-u(t)\|_{B} \ge \sigma \}| = 0 \,\, \forall \sigma >0.
\]

\begin{theorem}\label{theo:RossiSava}
    Let $\mathcal{U}$ be a family of measurable $B$-valued functions; if there exist a lower semicontinuous coercive integrand $\mathcal{F}: B \to [0,+\infty]$ and a distance $d:B \times B \to [0,+\infty]$, such that
    \[
    \mathcal{U} \text{ is tight w.r.t } \mathcal{F}, \text{ i.e. } \quad S:=\sup_{u \in \mathcal{U}} \int_0^T \mathcal{F}(u(t)) \, \dt < \infty,
    \]
    and
    \begin{equation}\label{eq:rossisav}
    \lim_{\tau \to 0} \sup_{u \in \mathcal{U}} \int_0^{T-\tau} d(u(t+\tau),u(t))\, \dt =0,
    \end{equation}
    then $\mathcal{U}$ is relatively compact in $\mathcal{M}(0,T;B)$.
\end{theorem}

\section{Technical tools} \label{appendixB}
In this section, we collect two technical lemmas from~\cite{Schaetzle} that are needed in Section~\ref{sec:compvarphieps} to establish the strong convergence of the phase-field functions as $\eps \to 0$.

\begin{lemma}{\cite[Lemma 5.1]{Schaetzle} }\label{lemma:Schaetzle}
For $\varphi \in W^{1, 2}(\Omega)$ and $C> 0$ with $\mathcal{E}_{\textnormal{int},\varepsilon}(\varphi) \le C$, there exists $\psi \in BV(\Omega; \{0,1\})$ satisfying $|\nabla \psi|(\Omega) \le C$ and
\[
\|\varphi - \psi\|_{L^2(\Omega)} \le \omega_0(\varepsilon),
\]
where $\omega_0 = \omega_{C, \Omega}$ satisfies $\lim_{\varepsilon \downarrow 0} \omega_0(\varepsilon) = 0$.
\end{lemma}

We provide the proof for the reader's convenience.

\begin{proof}
Assume, by contradiction, that the statement is false. Then there exist $\delta > 0$, a subsequence
$\varepsilon_j \downarrow 0$, and $(\varphi_j)_{j \in \mathbb{N}} \subset H^1(\Omega)$ such that
$\mathcal{E}_{\textnormal{int},\varepsilon_j}(\varphi_j) \leq C$ and
\begin{equation}
  \|\varphi_j - \psi\|_{L^2(\Omega)} \geq \delta
  \label{eq:assurdo}
\end{equation}
for every $j \in \mathbb{N}$ and $\psi \in BV(\Omega;\{0,1\})$ with $ |\nabla \psi|(\Omega) \leq C$. 

Define
\[
  \Psi(s) :=  \int_0^s \sqrt{2F(s)} \, ds
  \quad \text{ and } \quad \Psi_j := \Psi(\varphi_j),
\]
then we have that
$\|\Psi_j\|_{L^1(\Omega)} \leq C$ and
\[
  \int_\Omega |\nabla \Psi_j| \, \dx 
  =
   \int_\Omega \sqrt{2F(\varphi_j)} \, |\nabla \varphi_j|\, \dx 
  \leq \mathcal{E}_{\textnormal{int},\varepsilon_j}(\varphi_j)\leq C.
\]
Therefore, up to a subsequence, we deduce that
$(\Psi_j)_{j\in \mathbb{N}}$ converges strongly in $L^1(\Omega)$ and 
almost everywhere on $\Omega$ to a function $\bar \Psi \in BV(\Omega)$
with $\int_\Omega 1 \, \mathrm{d} |\nabla \bar \Psi|  \leq C$. 
Since $\Psi$ is
strictly increasing, $(\varphi_j)_{j \in \mathbb{N}}$
converges almost everywhere on $\Omega$ to
$ \chi= \Psi^{-1}(\bar \Psi)$. Moreover, from the uniform bound
$\|\varphi_j\|_{L^4(\Omega)} \leq C$ we obtain that
$\varphi_j \to \chi$ strongly in $L^2(\Omega)$. By Fatou's lemma, 
\[
  \int_\Omega F(\chi) \, \dx 
  \leq
  \liminf_{j \to \infty} \int_\Omega F(\varphi_j) \, \dx 
  \leq
  \liminf_{j \to \infty}
  \varepsilon_j \mathcal{E}_{\textnormal{int},\varepsilon_j}(\varphi_j)
  =
  0,
\]
hence $\chi \in \{0,1\}$. Since $\Psi(0) = 0$ and $\Psi(1) = \sigma$, it follows that 
$\bar \Psi = \Psi(\chi) = \sigma \chi$ with $\chi \in BV(\Omega;\{0,1\})$ and
$\int_\Omega 1\, \mathrm{d}|\nabla \chi| \leq C$.

Finally, taking $\psi = \chi$ in~\eqref{eq:assurdo} contradicts the strong convergence $\varphi_j \to \chi$  in $L^2(\Omega)$.
\end{proof}

\begin{lemma}{\cite[Section 5]{Schaetzle} (cf.~also~\cite{Luckhaus0})}
\label{lemma:Luckhaus}
    For $\psi \in L^1_0(\Omega;\{0,\pm 1\})$ and $\phi \in H^1(\Omega)$, we have
    \begin{equation*}
        \|\psi\|_{L^1_0(\Omega)} \le C \left(\|\psi + \phi \|_{L^1(\Omega)} + \|\psi + \phi \|^{1/2}_{L^1(\Omega)}\|\nabla \phi\|_{L^2(\Omega)}\right),
    \end{equation*}
    for some $C>0$ depending only on $\Omega$, where $L^1_0(\Omega)$ denotes the quotient space $L^1(\Omega)/\textnormal{span}\{1\}$.
\end{lemma}

\begin{remark} 
  We observe that, in case of an unbounded domain $\Omega = \mathbb{R}^d$, the same result follows by readapting the proof of~\cite[Lemma 10]{KLM} under the assumption $\rho, \tilde \rho \geq c >0$. In particular, the space $L_0^1(\Omega)=L^1(\Omega)/\operatorname{span}\{1\}$ can be naturally identified with
$L^1(\Omega)$, since the zero function is the only constant belonging to $L^1(\mathbb{R}^d)$.
\end{remark}

\section*{Acknowledgements}
The authors thank Tim Laux for suggesting the problem and some initial discussions, and Harald Garcke for bringing Sch\"atzle's paper to their attention.
A.K.\ research has been supported by the Austrian Science Fund (FWF) through grants 10.55776/F65, 10.55776/P35359, 10.55776/Y1292. 
A.M.\ research is sponsored by the Alexander von Humboldt Foundation, which is acknowledged. A.M.\ is funded by the Deutsche Forschungsgemeinschaft (DFG,
German Research Foundation) - CRC 1720 - 539309657.

\section*{AI usage statement}
The results in Section~\ref{sec:strongcompvarphi} were obtained with assistance from ChatGPT 5.6 Plus through a series of chats.
The idea of adapting the arguments of~\cite{Schaetzle} to our setting is due to the authors, and ChatGPT 5.6 Plus was used only to assist  with the implementation of this adaptation.
The authors subsequently checked, verified, and independently developed all mathematical arguments.
All the proofs contained in this manuscript were written by the authors, who take full responsibility for the correctness of the statements.
This article does not contain AI-written text.

\bibliographystyle{abbrv}
\bibliography{biblio}

@article{AbelsMoser,
  author  = {Abels, Helmut and Moser, Maximilian},
  title   = {Convergence of the {Allen--Cahn} equation to the mean curvature flow with $90^\circ$-contact angle in {2D}},
  journal = {Interfaces and Free Boundaries},
  volume  = {21},
  number  = {3},
  pages   = {313--365},
  year    = {2019},
  doi     = {10.4171/IFB/425}
}

@article{AlfaroGarckeHilhorstMatanoSchaetzle,
  author  = {Alfaro, Matthieu and Garcke, Harald and Hilhorst, Danielle and Matano, Hiroshi and Sch{\"a}tzle, Reiner},
  title   = {Motion by anisotropic mean curvature as sharp interface limit of an inhomogeneous and anisotropic {Allen--Cahn} equation},
  journal = {Proceedings of the Royal Society of Edinburgh: Section A Mathematics},
  volume  = {140},
  number  = {4},
  pages   = {673--706},
  year    = {2010},
  doi     = {10.1017/S0308210508000541}
}

@article{AlikakosBatesChen,
  author  = {Alikakos, Nicholas D. and Bates, Peter W. and Chen, Xinfu},
  title   = {Convergence of the {Cahn--Hilliard} equation to the {Hele--Shaw} model},
  journal = {Archive for Rational Mechanics and Analysis},
  volume  = {128},
  number  = {2},
  pages   = {165--205},
  year    = {1994},
  doi     = {10.1007/BF00375025}
}

@book {AFP,
	AUTHOR = {Ambrosio, L. and Fusco, N. and Pallara, D.},
	TITLE = {Functions of bounded variation and free discontinuity
	problems},
	SERIES = {Oxford Mathematical Monographs},
	PUBLISHER = {The Clarendon Press, Oxford University Press, New York},
	YEAR = {2000},
	PAGES = {xviii+434},
}

@book {AGS,
	AUTHOR = {Ambrosio, L. and Gigli, N. and Savar\'e, G.},
	TITLE = {Gradient flows in metric spaces and in the space of
	probability measures},
	SERIES = {Lectures in Mathematics ETH Z\"urich},
	EDITION = {Second},
	PUBLISHER = {Birkh\"auser Verlag, Basel},
	YEAR = {2008},
	PAGES = {x+334},
}

@article{CaginalpChen,
  author  = {Caginalp, Gunduz and Chen, Xinfu},
  title   = {Convergence of the phase field model to its sharp interface limits},
  journal = {European Journal of Applied Mathematics},
  volume  = {9},
  number  = {4},
  pages   = {417--445},
  year    = {1998},
  doi     = {10.1017/S0956792598003520}
}

@article{CahnHilliard,
  author  = {Cahn, John W. and Hilliard, John E.},
  title   = {Free Energy of a Nonuniform System. {I}. Interfacial Free Energy},
  journal = {The Journal of Chemical Physics},
  volume  = {28},
  number  = {2},
  pages   = {258--267},
  year    = {1958},
  doi     = {10.1063/1.1744102}
}

@article{CanariusSchaetzle,
title = {Finiteness and Positivity Results for Global Minimizers of a Semilinear Elliptic Problem},
journal = {Journal of Differential Equations},
volume = {148},
number = {1},
pages = {212--229},
year = {1998},
url = {https://doi.org/10.1006/jdeq.1998.3457},
author = {Canarius, T. and Sch\"{a}tzle, R.},
}

@article{Chambolle-Laux, 
  title={Mullins-{S}ekerka as the {W}asserstein flow of the perimeter},
  author={Chambolle, A. and Laux, T.},
  journal={Proc. Amer. Math. Soc.},
  volume  = {149},
  number  = {7},
  pages   = {2943-2956},
  year={2021},
}

@article{Chen,
  author  = {Chen, Xinfu},
  title   = {Generation and propagation of interfaces for reaction-diffusion equations},
  journal = {Journal of Differential Equations},
  volume  = {96},
  number  = {1},
  pages   = {116--141},
  year    = {1992},
  doi     = {10.1016/0022-0396(92)90146-E}
}

@article{Chen96,
  author  = {Chen, Xinfu},
  title   = {Global asymptotic limit of solutions of the Cahn-Hilliard equation},
  journal = {Journal of Differential Equations},
  volume  = {44},
  number  = {2},
  pages   = {262--311},
  year    = {1996},
  doi     = {10.4310/jdg/1214458973}
}

@article{DeMottoniSchatz,
  author  = {De Mottoni, Piero and Schatzman, Michelle},
  title   = {Geometrical evolution of developed interfaces},
  journal = {Trans. Amer. Math. Soc.},
  volume  = {347},
  number  = {5},
  pages   = {1533--1589},
  year    = {1995},
  doi     = {10.1090/S0002-9947-1995-1672406-7}
}

@article{ElbarPoiatti,
  author  = {Elbar, Charles and Poiatti, Andrea},
  title   = {Weak solutions and sharp interface limit of the anisotropic {Cahn--Hilliard} equation with disparate mobility and inhomogeneous potential},
  journal = {Interfaces and Free Boundaries},
  year    = {2026},
  doi     = {10.4171/IFB/573},
  note    = {Published online first}
}

@article{EvansSonerSouganidis,
  author  = {Evans, L. C. and Soner, H. M. and Souganidis, P. E.},
  title   = {Phase transitions and generalized motion by mean curvature},
  journal = {Communications on Pure and Applied Mathematics},
  volume  = {45},
  number  = {9},
  pages   = {109-7-1123},
  year    = {1992},
  doi     = {10.1002/cpa.3160450903}
}

@article{FischerLauxSimon,
  author  = {Fischer, Julian and Laux, Tim and Simon, Theresa M.},
  title   = {Convergence rates of the {Allen--Cahn} equation to mean curvature flow: A short proof based on relative entropies},
  journal = {SIAM Journal on Mathematical Analysis},
  volume  = {52},
  number  = {6},
  pages   = {6222--6233},
  year    = {2020},
  doi     = {10.1137/20M1322182}
}

@article{FischerMarveggio,
  author  = {Fischer, Julian and Marveggio, Alice},
  title   = {Quantitative convergence of the vectorial {Allen--Cahn} equation towards multiphase mean curvature flow},
  journal = {Annales de l'Institut Henri Poincar{\'e} C, Analyse Non Lin{\'e}aire},
  volume  = {41},
  number  = {5},
  pages   = {1117--1178},
  year    = {2024},
  doi     = {10.4171/AIHPC/109}
}

@article{GanediMarveggioStinson,
  author  = {Ganedi, Likhit and Marveggio, Alice and Stinson, Kerrek},
  title   = {Convergence of a heterogeneous {Allen--Cahn} equation to weighted mean curvature flow},
  journal = {Advances in Calculus of Variations},
  volume  = {18},
  number  = {4},
  pages   = {1301--1325},
  year    = {2025},
  doi     = {10.1515/acv-2024-0123}
}

@book{GilbargTrudinger,
  title={Elliptic partial differential equations of second order},
  author={Gilbarg, D. and Trudinger, N.},
  volume={224},
  year={2001},
  EDITION = {Second},
  publisher={Springer}
}

@article{HenselLauxcontact,
  author  = {Hensel, Sebastian and Laux, Tim},
  title   = {{BV} solutions for mean curvature flow with constant contact angle: {Allen--Cahn} approximation and weak-strong uniqueness},
  journal = {Indiana University Mathematics Journal},
  volume  = {73},
  number  = {1},
  pages   = {111--148},
  year    = {2024},
  doi     = {10.1512/iumj.2024.73.9701}
}

@article{HenselLauxvarifold,
  author  = {Hensel, Sebastian and Laux, Tim},
  title   = {A new varifold solution concept for mean curvature flow: Convergence of the {Allen--Cahn} equation and weak-strong uniqueness},
  journal = {Journal of Differential Geometry},
  volume  = {130},
  number  = {1},
  pages   = {209--268},
  year    = {2025},
  doi     = {10.4310/jdg/1747065796}
}

@article{Ilmanen,
  author  = {Ilmanen, Tom},
  title   = {Convergence of the {Allen--Cahn} equation to {Brakke}'s motion by mean curvature},
  journal = {Journal of Differential Geometry},
  volume  = {38},
  number  = {2},
  pages   = {417--461},
  year    = {1993},
  doi     = {10.4310/jdg/1214454300}
}

@article {JKO,
    AUTHOR = {Jordan, R. and Kinderlehrer, D. and Otto, F.},
     TITLE = {The variational formulation of the {F}okker-{P}lanck equation},
   JOURNAL = {SIAM J. Math. Anal.},
    VOLUME = {29},
      YEAR = {1998},
    NUMBER = {1},
     PAGES = {1--17},
       DOI = {10.1137/S0036141096303359},
       URL = {https://doi.org/10.1137/S0036141096303359},
}

@article{JKM_Muskat,
	author  = {Jacobs, Matt and Kim, Inwon and M\'esz\'aros, Alp\'ar R.},
	title   = {Weak solutions to the {M}uskat problem with surface tension via optimal transport},
	journal = {Archive for Rational Mechanics and Analysis},
	volume  = {239},
	number  = {1},
	pages   = {389--430},
	year    = {2021},
	doi     = {10.1007/s00205-020-01579-3}
}

@article{KroemerLaux,
  author  = {Kroemer, Milan and Laux, Tim},
  title   = {The {Hele--Shaw} flow as the sharp interface limit of the {Cahn--Hilliard} equation with disparate mobilities},
  journal = {Communications in Partial Differential Equations},
  volume  = {47},
  number  = {12},
  pages   = {2444--2486},
  year    = {2022},
  doi     = {10.1080/03605302.2022.2129384}
}

@article{KLM,
	AUTHOR = {Kubin, A. and Laux, T. and Marveggio, A.},
	TITLE = {The {V}erigin problem with phase transition as a {W}asserstein flow},
	JOURNAL = {Preprint},
	YEAR = {2026},
	URL = {https://arxiv.org/abs/2601.20001},
    NOTE = {arxiv:2601.20001},
}

@article{LauxSimon,
  author  = {Laux, Tim and Simon, Theresa M.},
  title   = {Convergence of the {Allen--Cahn} equation to multiphase mean curvature flow},
  journal = {Communications on Pure and Applied Mathematics},
  volume  = {71},
  number  = {8},
  pages   = {1597--1647},
  year    = {2018},
  doi     = {10.1002/cpa.21747}
}

@article{LauxStinson,
author = {Laux, Tim and Stinson, Kerrek},
title = {Sharp interface limit of the {C}ahn--{H}illiard reaction model for lithium-ion batteries},
journal = {Mathematical Models and Methods in Applied Sciences},
volume = {33},
number = {12},
pages = {2557-2585},
year = {2023},
doi = {10.1142/S0218202523500550},
}

@article{LauxTakasao,
	AUTHOR = {Laux, T. and Takasao, K.},
	TITLE = {Convergence of a vector-valued Allen-Cahn system to Brakke's multiphase mean curvature flow},
	JOURNAL = {Preprint},
	YEAR = {2026},
    note = {arXiv:2608.26842},
}

@article{LauxUllrichStinson,
  author  = {Laux, Tim and Stinson, Kerrek and Ullrich, Clemens},
  title   = {Diffuse-interface approximation and weak--strong uniqueness of anisotropic mean curvature flow},
  journal = {European Journal of Applied Mathematics},
  volume  = {36},
  number  = {1},
  pages   = {82--142},
  year    = {2025},
  doi     = {10.1017/S0956792524000226}
}

@article{LuckhausModica,
	author  = {Luckhaus, Stephan and Modica, Luciano},
	title   = {The {G}ibbs--{T}hompson relation within the gradient theory of phase transitions.},
	journal = {Archive for Rational Mechanics and Analysis},
	volume  = {107},
	number  = {1},
	pages   = {71--83},
	year    = {1989},
	doi     = {10.1007/BF00251427}
}

@article{Luckhaus1,
	author  = {Luckhaus, Stephan},
	title   = {Solutions for the two-phase {S}tefan problem with the {G}ibbs-{T}homson law for the melting temperature},
	journal = {European Journal of Applied Mathematics},
	volume  = {1},
	number  = {2},
	pages   = {101--111},
	year    = {1990},
	doi     = {10.1017/S0956792500000103}
}

@techreport{Luckhaus0,
	author      = {Luckhaus, Stephan},
	title       = {The {S}tefan problem with {G}ibbs-{T}homson law},
	institution = {Sezione di Analisi Matematica e Probabilit\`{a}, Universit\`{a} di Pisa},
	type        = {Report},
	number      = {2.75 (591)},
	year        = {1991}
}

@article{LuckStur,
	author  = {Luckhaus, Stephan and Sturzenhecker, Thomas},
	title   = {Implicit Time Discretization for the Mean Curvature Flow Equation},
	journal = {Calculus of Variations and Partial Differential Equations},
	volume  = {3},
	number  = {2},
	pages   = {253--271},
	year    = {1995},
	doi     = {10.1007/BF01205007}
}

@article{MielkeRossiSavare,
	author = {Mielke, A. and Rossi, R. and Savar\'e, G.},
	title = {Nonsmooth analysis of doubly nonlinear evolution equations},
	journal = {Calculus of Variations and Partial Differential Equations},
	pages = {253--310},
	volume = {46},
	number = {1},
	year = {2013},
	url = {https://doi.org/10.1007/s00526-011-0482-z}
}

@article{MizunoTonegawa,
  author  = {Mizuno, Masashi and Tonegawa, Yoshihiro},
  title   = {Convergence of the {Allen--Cahn} equation with {Neumann} boundary conditions},
  journal = {SIAM Journal on Mathematical Analysis},
  volume  = {47},
  number  = {3},
  pages   = {1906--1932},
  year    = {2015},
  doi     = {10.1137/140987808}
}

@article {ModicaMortola,
    AUTHOR = {Modica, L. and Mortola, S.},
     TITLE = {Un esempio di {$\Gamma$}-convergenza},
   JOURNAL = {Boll. Un. Mat. Ital. B (5)},
  FJOURNAL = {Unione Matematica Italiana. Bollettino. B. Serie V},
    VOLUME = {14},
      YEAR = {1977},
    NUMBER = {1},
     PAGES = {285--299},
}

@article{Otto1,
	author  = {Otto, Felix},
	title   = {The geometry of dissipative evolution equations: the porous medium equation},
	journal = {Communications in Partial Differential Equations},
	volume  = {26},
	number  = {1-2},
	pages   = {101--174},
	year    = {2001},
	doi     = {10.1081/PDE-100002243}
}

@article{Piccoli-Rossi,
title = {Generalized Wasserstein Distance and its Application to Transport Equations with Source},
journal = {Archive for Rational Mechanics and Analysis},
volume = {211},
number = {1},
pages = {335--358},
year = {2014},
doi = {https://doi.org/10.1007/s00205-013-0669-x},
author = {Piccoli, Benedetto and Rossi, Francesco},
}

@article{PrussSim,
	author  = {Pr\"uss, Jan and Simonett, Gieri},
	title   = {The {V}erigin problem with and without phase transition},
	journal = {Interfaces and Free Boundaries},
	volume  = {20},
	number  = {1},
	pages   = {107--128},
	year    = {2018},
	doi     = {10.4171/IFB/398}
}

@article{PrussSimWielke,
	author  = {Pr{\"u}ss, Jan and Simonett, Gieri and Wilke, Mathias},
	title   = {The {R}ayleigh--{T}aylor instability for the {V}erigin problem with and without phase transition},
	journal = {Nonlinear Differential Equations and Applications},
	volume  = {26},
	number  = {3},
	pages   = {Art. 18, 35 pp.},
	year    = {2019},
	doi     = {10.1007/s00030-019-0564-8}
}

@article{RogerTonegawa,
  author  = {R{\"o}ger, Matthias and Tonegawa, Yoshihiro},
  title   = {Convergence of phase-field approximations to the {Gibbs--Thomson} law},
  journal = {Calculus of Variations and Partial Differential Equations},
  volume  = {32},
  number  = {1},
  pages   = {111--136},
  year    = {2008},
  doi     = {10.1007/s00526-007-0133-6}
}

@article{Rossi-Savare,
	author = {Rossi, R. and Savar\'e, G.},
	title = {Tightness, integral equicontinuity and compactness for evolution problems in {Banach} spaces},
	journal = {Annali della Scuola Normale Superiore di Pisa - Classe di Scienze},
	pages = {395--431},
	publisher = {Scuola normale superiore},
	volume = {Ser. 5, 2},
	number = {2},
	year = {2003},
	url = {https://www.numdam.org/item/ASNSP_2003_5_2_2_395_0/}
}

@article{RossiSavare2006,
	author = {Rossi, R. and Savar\'e, G.},
	title = {Gradient flows of non convex functionals in {Hilbert} spaces and applications},
	journal = {ESAIM: COCV},
	pages = {564--614},
	volume = {12},
	number = {3},
	year = {2006},
	url = {https://doi.org/10.1051/cocv:2006013}
}

@article{SandierSerfaty,
  author  = {Sandier, Etienne and Serfaty, Sylvia},
  title   = {{Gamma}-convergence of gradient flows with applications to {Ginzburg--Landau}},
  journal = {Communications on Pure and Applied Mathematics},
  volume  = {57},
  number  = {12},
  pages   = {1627--1672},
  year    = {2004},
  doi     = {10.1002/cpa.20046}
}

@book{Santabook, 
	title={Optimal Transport for Applied Mathematicians},
	author={Santambrogio, F.},
	publisher={Birkh\"auser Cham},
	year={2016}
}

@article{Schaetzle,
title = {The Quasistationary Phase Field Equations with Neumann Boundary Conditions},
journal = {Journal of Differential Equations},
volume = {162},
number = {2},
pages = {473-503},
year = {2000},
doi = {https://doi.org/10.1006/jdeq.1999.3679},
url = {https://www.sciencedirect.com/science/article/pii/S0022039699936793},
author = {Reiner Sch\"{a}tzle},
}

@article{Steinke,
  author  = {Steinke, Pascal},
  title   = {Convergence of {Allen--Cahn} equations to {De Giorgi}'s multiphase mean curvature flow},
  journal = {Journal of Evolution Equations},
  volume  = {26},
  pages   = {66},
  year    = {2026},
  doi     = {10.1007/s00028-026-01203-z}
}

\end{document}